\documentclass[12pt,a4paper]{article}
\usepackage{amsthm}
\usepackage{amsmath}
\usepackage{amsfonts} 
\usepackage{enumerate}
\usepackage[shortlabels]{enumitem}
\usepackage{calc}
\usepackage[all]{xy}
\usepackage[colorlinks=true, citecolor=blue, linkcolor=red, pagebackref=true]{hyperref}
\usepackage{comment}

\AtBeginDocument{%
\setlength{\abovedisplayskip}{10pt plus 3pt minus 6pt}
\setlength{\belowdisplayskip}{10pt plus 3pt minus 6pt}
\setlength{\abovedisplayshortskip}{0pt plus 3pt}
\setlength{\belowdisplayshortskip}{4pt plus 3pt minus 4pt}
}

\theoremstyle{plain}
\newtheorem {lemma}{Lemma}[section]
\newtheorem {proposition}[lemma]{Proposition}
\newtheorem {theorem}[lemma]{Theorem}
\newtheorem {corollary}[lemma]{Corollary}

\theoremstyle{definition}
\newtheorem {definition}[lemma]{Definition}
\newtheorem {remark}[lemma]{Remark}
\newtheorem {example}[lemma]{Example}

\def\N{\mathbb N}
\newcommand{\Mat}{\operatorname{\mathbb{M}}}
\newcommand{\T}{\operatorname{T}}
\newcommand{\diag}{\operatorname{diag}}
\newcommand{\GL}{\operatorname{GL}}
\newcommand{\Gg}{\operatorname{G}}
\newcommand{\Ef}{\operatorname{F}}
\newcommand{\Ee}{\operatorname{E}}
\newcommand{\De}{\operatorname{D}}
\newcommand{\Pe}{\operatorname{P}}
\newcommand{\Ob}{\operatorname{Ob}}
\newcommand{\Hom}{\operatorname{Hom}}
\newcommand{\Mor}{\operatorname{Mor}}
\newcommand{\reg}{\operatorname{reg}}
\newcommand{\sink}{\operatorname{sink}}
\newcommand{\id}{\operatorname{id}}

\def\0{{\mathbf 0}}

\def\x{{\bf x}}

\def\b{{\bf b}}

\def\e{{\bf e}}
\def\B{{\mathcal{B}}}
\def\C{{\mathcal{C}}}

\def\E{{\mathcal{E}}}
\def\I{{\mathcal{I}}}

\def\Es{{\mathcal{S}}}
\def\V{{\mathcal{V}}}

\def\Ge{{\mathbf{G}}}
\def\F{{\mathbf{F}}}
\def\EE{{\mathbf{E}}}
\def\H{{\mathbf{H}}}
\def\P{{\mathbf{P}}}

\title{The general linear groupoid of a Leavitt path algebra}
\author{Raimund Preusser}
\date{}
\AtEndDocument{\bigskip{\footnotesize%
  \textsc{School of Mathematics and Statistics, Nanjing University of Information Science and Technology, Nanjing, China} \par  
  \textit{E-mail address:} \texttt{raimund.preusser@gmx.de} \par
}}

\begin{document}
\maketitle
\begin{abstract}
\noindent
The general linear groupoid $\Ge(L(E))$ of a Leavitt path algebra $L(E)$ is isomorphic to the groupoid whose objects are all $L(E)$-modules of the form $\bigoplus_{i=1}^n v_iL(E)$ where each $v_i$ is a vertex, and whose morphisms are all isomorphisms between these modules. We find a generating set for $\Ge(L(E))$ and consequently obtain generating sets for all general linear groups $\GL_n(L(E))$ over $L(E)$ (including the group $\GL_1(L(E))$ of invertible elements of $L(E)$). We prove similar results for Leavitt path algebras of hypergraphs (which generalise the Leavitt path algebras of separated graphs and vertex-weighted graphs).
\end{abstract}
\let\thefootnote\relax
\let\thefootnote\relax
%\tableofcontents

\section{Introduction}
Leavitt path algebras are algebras associated to graphs. They were introduced by Abrams and Aranda Pino \cite{AMP} and independently by Ara, Moreno and Pardo \cite{AA}. The definition of a Leavitt path algebra was inspired on the one hand by Leavitt's algebras $L(1,n)$ of module type $(1,n)$, and on the other hand by the graph $C^*$-algebras. The field of Leavitt path algebras is a very active research area and has connections to many other branches of mathematics,
such as functional analysis, symbolic dynamics, K-theory and noncommutative geometry. We refer the reader to the book \cite{lpabook} for more information about these algebras.

Despite the extensive literature on Leavitt path algebras, their invertible elements and general linear groups are poorly understood. In this paper, we find for every finite graph $E$ a generating set for the general linear groupoid $\Ge(L(E))$ of the Leavitt path algebra $L(E)$ with coefficients in a field $K$. The groupoid $\Ge(L(E))$ is isomorphic to the groupoid whose objects are all $L(E)$-modules of the form $\bigoplus_{i=1}^n v_iL(E)$ where each $v_i$ is a vertex, and whose morphisms are all isomorphisms between these modules. Since each general linear group $\GL_n(L(E))~(n\geq 1)$ is isomorphic to an isotropy group of $\Ge(L(E))$, we consequently obtain generating sets for all general linear groups over $L(E)$. 

We prove similar results for Leavitt path algebras of hypergraphs. These algebras were introduced in \cite{Raimund2}, they simultaneously generalise the Leavitt path algebras of separated graphs and vertex-weighted graphs. Leavitt path algebras of separated graphs were introduced by Ara and Goodearl \cite{aragoodearl} in an attempt to provide a class of algebras that generalises the usual Leavitt path algebras and covers all Leavitt algebras $L(m,n)$ (the usual Leavitt path algebras only cover the Leavitt algebras $L(1,n)$). With the same goal Hazrat defined weighted Leavitt path algebras \cite{H-1}. The Leavitt path algebras of vertex-weighted graphs form a subclass of the class of weighted Leavitt path algebras, which is still large enough to embrace the usual Leavitt path algebras as well as all algebras $L(m,n)$. Abrams and Hazrat demonstrated that Leavitt path algebras of vertex-weighted graphs are closely connected to sandpile models and chip-firing games \cite{AH}.

The rest of the paper is organised as follows.

In Section 2, we consider rings $R$ with a fixed complete set of orthogonal idempotents $\I=\{e^{(1)},\dots,e^{(p)}\}$. For $\sigma\in\I^n$ define the finitely generated projective $R$-module $R^{\sigma}=\bigoplus_{i=1}^n\sigma_iR$. If $\sigma\in\I^m$ and $\tau\in \I^n$, then the $R$-linear maps between $R^\sigma$ and $R^\tau$ are in 1-1 correspondence to matrices $A\in\Mat_{n\times m}(R)$ such that $A_{ij}\in\tau_iR\sigma_j$ for all $i$ and $j$. We denote the set of all such matrices by $\Mat_{\tau\times\sigma}(R)$.

We denote by $\GL_{\tau\times\sigma}(R)$ the set of all matrices in $\Mat_{\tau\times\sigma}(R)$ that define an isomorphism between $R^\sigma$ and $R^\tau$. The \textit{general linear groupoid} $\Ge(R)$ is the groupoid with $\Ob(\Ge(R))=\bigcup_{n\geq 1}\I^n$ and $\Hom_{\Ge(R)}(\sigma,\tau)=\GL_{\tau\times \sigma}(R)$ for all objects $\sigma$ and $\tau$ (composition of morphisms is matrix multiplication). The groups $\GL_\sigma(R)=\GL_{\sigma\times\sigma}(R)$ are the isotropy groups of $\Ge(R)$. They embrace the usual general linear groups of $R$. %The elementary and diagonal subgroups $\Ee_\sigma(R)$ and $\De_\sigma(R)$ are defined in the obvious way.

In Section 3, we prove a version of Bergman's result \cite[Corollary 2.15]{bergman74} for the case that $R_0$ is a finite product of skew fields. Namely we show the following. Let
\begin{itemize}
\item $R_0=K^{(1)}\times\dots \times K^{(p)}$ where $K^{(1)},\dots,K^{(p)}$ are skew fields,
\item $(R_\lambda)_{\lambda\in\Lambda}$ be a family of faithful $R_0$-rings,
\item $R$ be the coproduct of the $R_\lambda$ in the
category of $R_0$-rings.
\end{itemize}
Moreover, let $\Ge(R)$ be the general linear groupoid of $R$ with respect to $\I=\{e^{(1)},\dots,e^{(p)}\}$, where $e^{(j)}$ denotes the element of $R_0$ with $1$ in the $j$-th place and $0$'s elsewhere. Theorem \ref{thm:main} asserts that if Condition (P) in Definition \ref{def:condP} is satisfied, then $\Ge(R)$ is generated by the subgroupoids $\Ge(R,\mu)$ where $\mu$ ranges over $\Lambda\cup\{0\}$. Here $\Ge(R,\mu)$ denotes the subgroupoid of $\Ge(R)$ generated by the sets $\GL_{\tau\times\sigma}(R_\mu)$ and the $\mu$-based transvections (see Definition \ref{def:transvection}). We consequently obtain generating sets for all general linear groups $\GL_\sigma(R)$. 

In Section 4, we apply the main result from Section 3 to Leavitt path algebras $L(E)$ of finite graphs $E$ with coefficients in a field $K$. First we show that there are integers $m$ and $n$ such that $\Mat_n(L(E))\cong R_1\amalg_{R_0} R_2$ where $R_0=K^m$ and $R_1$ and $R_2$ are $R_0$-rings. Then we apply Theorem~\ref{thm:main} to obtain a generating set for the groupoid $\Ge(\Mat_n(L(E)))$. Now one can use the fact that there is a surjective functor $\Ge(\Mat_n(L(E)))\to \Ge(L(E))$, which is induced by the Morita functor $\operatorname{Mod}(\Mat_n(L(E)))\longrightarrow \operatorname{Mod}(L(E))$, to obtain a generating set for $\Ge(L(E))$ (here $\Ge(L(E))$ denotes the general linear groupoid of $L(E)$ with respect to the vertex set $E^0$). We consequently obtain generating sets for all general linear groups $\GL_\sigma(L(E))$. 

In Section 5, we apply the main result from Section 3 to Leavitt path algebras $L(H)$ of finite hypergraphs $H$ with coefficients in a field $K$. By doing so we obtain a generating set for the general linear groupoid $\Ge(L(H))$ (with respect to the vertex set $H^0$), and consequently generating sets for all general linear groups $\GL_\sigma(L(H))$. 

\section{Linear algebra over rings with a complete set of orthogonal idempotents}
In this paper, all rings will be associative and unital unless otherwise stated. All homomorphisms and modules will be assumed unital, and module will mean right module.

In this section, $R$ denotes a ring and $\I=\{e^{(1)},\dots,e^{(p)}\}$ a complete set of orthogonal idempotents in $R$ (i.e. $e^{(i)}e^{(j)}=\delta_{ij}e^{(i)}~(1\leq i,j\leq p)$ and $e^{(1)}+\dots +e^{(p)}=1$).
\subsection{Pseudo-free modules}

\begin{definition}\label{def:pseudobasis}
Let $M$ be an $R$-module. A finite subset $\B\subseteq M$ is called a \textit{pseudo-basis} for $M$ if for each $1\leq j\leq p$ there is a subset $\B_j\subseteq Me^{(j)}$ such that
\begin{enumerate}[(i)]
\item $\B$ is the disjoint union of the $\B_j$, and
\item every element $\x\in M$ can be uniquely written as a linear combination $\x=\sum_{b\in \B}\b r_{\b}$ where almost all $r_{\b}$ are zero and $r_{\b}\in e^{(j)}R$ if $\b\in \B_j$. 
\end{enumerate}
If an $R$-module $M$ has a pseudo-basis, we say that $M$ is \textit{pseudo-free}.
\end{definition}

Example~\ref{ex:mat3} below shows that an $R$-module can have pseudo-bases with different cardinalities.

\begin{example}\label{ex:mat3}
Suppose that $R=\Mat_3(S)$ for some ring $S$, and let $e_{ij}~(1\leq i,j\leq 3)$ denote the standard matrix units in $R$. Clearly $e^{(1)}:=e_{11}$ and $e^{(2)}:=e_{22}+e_{33}$ form a complete set of orthogonal idempotents in $R$. Let $M=e^{(2)}R$. One checks easily that
\[\B=\{e_{22}+e_{33}\}\quad\text{ and }\quad \C=\{e_{21},e_{31}\}\]
\smallskip
are pseudo-bases for $M$ (here $\B_1=\emptyset$, $\B_2=\{e_{22}+e_{33}\}$, $\C_1=\{e_{21},e_{31}\}$ and $\C_2=\emptyset$).
\end{example}

\begin{definition}
A pseudo-basis is called \textit{ordered} if it is equipped with an ordering of its elements. Suppose $\B$ is an ordered pseudo-basis with elements $\b_1,\dots,\b_n$ (in this order) such that $\b_i\in \B_{j_i}~(1\leq i\leq n)$. Then $\sigma=(e^{(j_1)},\dots,e^{(j_n)})\in \I^n$ is called the \textit{signature} of $\B$.
\end{definition}

\begin{example}\label{ex:Rn}
Let $\sigma\in \I^n$ where $n\geq 1$. We denote by $R^\sigma$ the submodule 
\[R^\sigma=\{\x=(x_1,\dots,x_{n})^T\in R^{n}\mid x_i\in \sigma_i R \text{ for all }1\leq i\leq n\}\] of the free $R$-module $R^{n}$. Clearly \[\E=\{\e^{(1)},\e^{(2)},\dots,\e^{(n)}\}\]
is a pseudo-basis for $R^\sigma$, where for $1\leq i\leq n$, $\e^{(i)}\in R^\sigma$ is the vector whose $i$-th component is $\sigma_i$ and whose other components are $0$. We view $\E$ as ordered pseudo-basis with signature $\sigma$ in the obvious way.
\end{example}

We call the $R$-module $R^{\sigma}$ defined in Example~\ref{ex:Rn} the \textit{pseudo-free $R$-module of signature $\sigma$}. Note that $R^\sigma$ is a direct summand of the free $R$-module $R^{n}$. Hence $R^\sigma$ is a finitely generated projective $R$-module.

\subsection{Linear maps between pseudo-free modules}

\begin{definition}\label{def:Mat}
Let $\sigma\in \I^m$ and $\tau \in \I^n$. We denote by $\Mat_{\sigma\times \tau}(R)$ the set of all $m\times n$-matrices $A$ such that $A_{ij}\in \sigma_iR\tau_j$ for any $1\leq i\leq m$ and $1\leq j\leq n$.
\end{definition}
We define the set $\Es=\Es_\I:=\bigcup_{n\geq 1}\I^n$. %If $\sigma\in \I^n\subseteq \Es$, then we set $|\sigma|:=n$.
If $\sigma,\tau\in \Es$ and $A\in \Mat_{\tau\times \sigma}(R)$, then we denote by $\phi_A$ the $R$-linear map $R^\sigma\to R^\tau$ which maps $\x\mapsto A\x$.

\begin{lemma}\label{lem:linmap}
Let $\sigma,\tau\in \Es$ and $\phi:R^\sigma\to R^\tau$ an $R$-linear map. Then there is a unique matrix $A\in \Mat_{\tau\times \sigma}(R)$ such that $\phi=\phi_A$ (namely the matrix $A$ whose $j$-th column is $\phi(\e^{(j)})$).
%, where $\E=\{\e^{(1)},\e^{(2)},\dots,\e^{(\hat\m)}\}$ is the standard pseudo-basis for $R^{\m}$).
\end{lemma}
\begin{proof}
Straightforward.
\end{proof}

If $\B$ is an ordered pseudo-basis for $M$ with signature $\sigma\in\Es$, then the \textit{coordinate map}
\[[\cdot]_\B:M\to R^\sigma\]
which maps each $\x\in M$ to its coordinate vector $[\x]_\B$ with respect to $\B$ is an $R$-module isomorphism. Hence every nonzero pseudo-free $R$-module is isomorphic to $R^\sigma$ for some $\sigma\in \Es$, and therefore is finitely generated projective.

\begin{proposition}\label{prop:matrep}
Let $M$ and $N$ be $R$-modules. Let $\B$ be an ordered pseudo-basis for $M$ with signature $\sigma$, and $\C$ an ordered pseudo-basis for $N$ with signature $\tau$. Then for any $R$-linear map $\phi:M\to N$ there is a unique matrix $A\in \Mat_{\tau\times \sigma}(R)$ such that the diagram
\[
\xymatrix@C=1cm@R=1cm{
M\ar[r]^{\phi} \ar[d]_{[\cdot]_\B} & N \ar[d]^{[\cdot]_{\C}} \\
R^\sigma \ar[r]_{\phi_A} & R^\tau
}
\]
commutes (namely the matrix $A$ whose $j$-th column is $[\phi(\b^{(j)})]_\C$, where $\b^{(j)}$ is the $j$-th element of $\B$).
\end{proposition}
\begin{proof}
The proposition follows from Lemma~\ref{lem:linmap}.
\end{proof}

The matrix $A$ in Proposition~\ref{prop:matrep} is called the \textit{matrix representation} of $\phi$ with respect to the pseudo-bases $\B$ and $\C$. 

\subsection{General linear groups and the general linear groupoid}

\begin{definition}
Let $\sigma,\tau\in \Es$. We call a matrix $A\in\Mat_{\sigma\times \tau}(R)$ \textit{invertible} if there is a matrix $B\in \Mat_{\tau\times \sigma}(R)$ such that 
\[AB=I_\sigma\quad\text{and}\quad BA=I_\tau\]
where for $\rho\in \I^n$, $I_\rho=\diag(\rho_1,\rho_2,\dots,\rho_{n})\in \Mat_{\rho\times \rho}(R)$. 
If $A$ is invertible, then we call the uniquely determined matrix $B$ with the property above the \textit{inverse} of $A$ and denote it by $A^{-1}$. We denote the subset of $\Mat_{\sigma\times \tau}(R)$ consisting of all invertible matrices by $\GL_{\sigma\times \tau}(R)$. Instead of $\GL_{\sigma\times \sigma}(R)$ we will usually write $\GL_{\sigma}(R)$. The group $\GL_{\sigma}(R)$ is called the \textit{general linear group} corresponding to $\sigma$. 
\end{definition}

The groups $\GL_{\sigma}(R)~(\sigma\in \Es)$ cover the usual general linear groups over $R$, as the proposition below shows.

\begin{proposition}\label{prop:iso}
If $\sigma=(\underbrace{e^{(1)},\dots,e^{(1)}}_{n\text{ terms}},\underbrace{e^{(2)},\dots,e^{(2)}}_{n\text{ terms}},\dots,\underbrace{e^{(p)},\dots,e^{(p)}}_{n\text{ terms}})\in \Es$, then $\GL_{\sigma}(R)\cong \GL_{n}(R)$.
\end{proposition}
\begin{proof}
We can write each $A\in \GL_{\sigma}(R)$ in the form 
\[A=\begin{pmatrix}
^{(1}\!A^{1)}&\dots&^{(1}\!A^{p)}\\
\vdots&\ddots&\vdots\\
^{(p}\!A^{1)}&\dots&^{(p}\!A^{p)}
\end{pmatrix}\]
where each $^{(i}\!A^{j)}$ is an $n\times n$ matrix whose entries lie in $e^{(i)}Re^{(j)}$.
There is an isomorphism $\GL_{\sigma}(R)\to\GL_{n}(R)$ mapping 
\[A
\longmapsto
\sum_{i,j=1}^p{}^{(i}\!A^{j)}.
\]
Its inverse maps
\[B\longmapsto\begin{pmatrix}
e^{(1)}Be^{(1)}&\dots&e^{(1)}Be^{(p)}\\
\vdots&\ddots&\vdots\\
e^{(p)}Be^{(1)}&\dots&e^{(p)}Be^{(p)}
\end{pmatrix}.
\]
\end{proof}

\begin{definition}\label{def:catP}
We denote by $\P(R)=\P_\I(R)$ the category with
\begin{itemize}
\item $\Ob(\P(R))=\Es$,
\item $\Hom_{\P(R)}(\sigma,\tau)=\Mat_{\tau\times \sigma}(R)$ for all $\sigma,\tau\in \Es$,
\item composition of morphisms is matrix multiplication, and
\item the identity morphism of an object $\sigma$ is the matrix $I_\sigma$.
\end{itemize}
\end{definition}

\begin{remark}\label{rem:xi}
Let $\P'(R)$ be the category whose objects are the $R$-modules
$R^\sigma$ $(\sigma\in\Es)$ and whose morphisms are all $R$-module
homomorphisms between those modules. There is a canonical isomorphism
\[
\xi:\P(R)\longrightarrow \P'(R)
\]
of categories, sending an object $\sigma$ to the module $R^\sigma$
and a morphism $A\in\Mat_{\tau\times\sigma}(R)$ to the left
multiplication map $\phi_A:R^\sigma\to R^\tau$ defined by $\phi_A(\x)=A\x$.
\end{remark}

Recall that a \textit{groupoid} is a small category in which every morphism is invertible. We will often identify a groupoid with the set of its morphisms. 

\begin{definition}
The subcategory $\Ge(R)=\Ge_\I(R)$ of $\P(R)$ such that 
\begin{itemize}
\item $\Ob(\Ge(R))=\Ob(\P(R))=\Es$,
\item $\Hom_{\Ge(R)}(\sigma,\tau)=\GL_{\tau\times \sigma}(R)$ for all $\sigma,\tau\in \Es$
\end{itemize}
is called the \textit{general linear groupoid of $R$ (with respect to $\I$)}.
\end{definition}

\begin{remark}\label{Gcomp}
Suppose the objects $\sigma,\tau\in \Es$ lie in the same connected component of $\Ge(R)$, i.e. there exists an $A\in\GL_{\tau\times \sigma}(R)$. Then the map 
\begin{align*}
\GL_\sigma(R)&\to\GL_\tau(R)\\
B&\mapsto ABA^{-1}
\end{align*}
is a group isomorphism. Hence all general linear groups corresponding to the objects of a connected component of $\Ge(R)$ are isomorphic.
\end{remark}

\subsection{Subgroups of general linear groups}

\begin{definition}
Let $\sigma\in \Es$. The subgroup $\De_{\sigma}(R)$ of $\GL_{\sigma}(R)$ consisting
of all diagonal matrices in $\GL_{\sigma}(R)$ (i.e.\ matrices $A$ with $A_{ij}=0$ for $i\neq j$)
is called the \textit{diagonal subgroup}.
\end{definition}

\begin{definition}
Let $\sigma\in \I^n$. Moreover, let $1\le i,j\le n$ with $i\neq j$, and $x\in \sigma_i R \sigma_j$. Then the matrix 
\[E_{ij}(x)=I_\sigma + e_{ij}(x)\in\GL_{\sigma}(R),\]
where $e_{ij}(x)$ is the $n\times n$ matrix whose entry at position $(i,j)$ is $x$ and all other entries are $0$, is called an \textit{elementary matrix}.
The subgroup $\Ee_{\sigma}(R)$ of $\GL_{\sigma}(R)$ generated by all elementary matrices is called the \textit{elementary subgroup}.
\end{definition}

\begin{definition}\label{def:transvection}
Let $\sigma\in \Es$. If $A\in \Mat_{\sigma\times \tau}(R)$, $B\in \Mat_{\tau\times \tau}(R)$, $C\in \Mat_{\tau\times \sigma}(R)$ for some $\tau\in \Es$, and $CA=0$, then the matrix
\[I_\sigma+ABC\in \GL_{\sigma}(R)\]
is called a \textit{transvection}. We denote by $\T_\sigma(R)$ the subgroup of $\GL_\sigma(R)$ generated by all transvections.
\end{definition}

%Note that $(E_{ij}(x))^{-1}=E_{ij}(-x)$ and $(I_\sigma+ABC)^{-1}=I_\sigma-ABC$.

\begin{definition}
Suppose that $R$ is a $K$-algebra over a field $K$. For $\sigma,\tau\in\Es$ we denote
\begin{itemize}
\item by $\Mat_{\sigma\times\tau}(K)$ the set consisting of all $A\in\Mat_{\sigma\times \tau}(R)$ such that $A_{ij}\in \sigma_iK\tau_j$ for all $i,j$,
\item by $\GL_{\sigma\times\tau}(K)$ the set
      $\GL_{\sigma\times\tau}(R)\cap\Mat_{\sigma\times \tau}(K)$,
\item by $\GL_{\sigma}(K)$ the group $\GL_{\sigma}(R)\cap \Mat_{\sigma\times\sigma}(K)$,
\item by $\De_{\sigma}(K)$ the group $\De_{\sigma}(R)\cap \Mat_{\sigma\times\sigma}(K)$,
\item by $\Ee_{\sigma}(K)$ the subgroup of $\GL_{\sigma}(K)\cap\Ee_{\sigma}(R)$
generated by all elementary matrices $E_{ij}(x)$ with
$x\in \sigma_i K \sigma_j$.
\end{itemize}
\end{definition}

%\begin{remark}\label{rem:block}
%Note that $\sigma_i K \sigma_j$ is nonzero only when $\sigma_i=\sigma_j$. Hence every %$A\in\GL_{\sigma}(K)$ is block‑diagonal with respect to the blocks determined by $\sigma$. Thus %$\GL_\sigma(K)$ is isomorphic to $\GL_{n_1}(K)\times\dots\times \GL_{n_p}(K)$.
%\end{remark}

\begin{definition}
Let $\sigma,\tau\in \I^n$ and $\pi$ a permutation such that $\sigma_i=\tau_{\pi(i)}$ for all $1\leq i\leq n$. A matrix $P\in \GL_{\sigma\times\tau}(K)$ is called a \textit{(generalised) permutation matrix} if there are $k_1,\dots,k_n\in K^\times$ such that 
\[P_{ij}=\delta_{j,\pi(i)}k_i\sigma_i\]
for all $1\leq i,j\leq n$. We denote the set of all permutation matrices in $\GL_{\sigma\times\tau}(K)$ by $\Pe_{\sigma\times\tau}(K)$. 
\end{definition}

Note that if $\sigma=\tau$ and $\pi=\id$, then $\Pe_{\sigma\times\tau}(K)=\De_\sigma(K)$.

\begin{definition}\label{def:ordered}
We call an element $\sigma\in \Es$ \textit{ordered} if it has the form 
\[\sigma=(\underbrace{e^{(1)},\dots,e^{(1)}}_{n^\sigma_1\text{ terms}},\underbrace{e^{(2)},\dots,e^{(2)}}_{n^\sigma_2\text{ terms}},\dots,\underbrace{e^{(p)},\dots,e^{(p)}}_{n^\sigma_p\text{ terms}})\]
where $n^\sigma_1,\dots,n^\sigma_p$ are nonnegative integers.
\end{definition}

Let $\Ge$ be a groupoid and $S\subseteq\Mor(\Ge)$ a set of morphisms. Recall that the subgroupoid $\H$ of $\Ge$ such that the objects of $\H$ are all objects of $\Ge$ that appear as the source or target of a morphism in $S$, and the morphisms of $\H$ are all finite compositions of morphisms from $S\cup S^{-1}$ is called the \textit{subgroupoid of $\Ge$ generated by $S$}.

\begin{lemma}\label{lem:GE}
Suppose that $R$ is a $K$-algebra over a field $K$, and let $\sigma\in \Es$. Then $\GL_{\sigma}(K)$ is contained in the subgroupoid of $\Ge(R)$ generated by the sets $\Ee_{\kappa}(K)$ and $\Pe_{\kappa\times\lambda}(K)$ where $\kappa,\lambda\in\Es$.
\end{lemma}
\begin{proof}
Let $A\in \GL_{\sigma}(K)$. We may assume that $\sigma$ is ordered (conjugate $A$ by an appropriate permutation matrix). Note that $\sigma_i K \sigma_j$ is nonzero only when $\sigma_i=\sigma_j$. Hence $\GL_\sigma(K)$ is isomorphic to $\GL_{n^\sigma_1}(K)\times\dots\times \GL_{n^\sigma_p}(K)$. The lemma follows now from the well-known fact that $\GL_n(K)=\De_{n}(K)\Ee_{n}(K)$ for any positive integer $n$.
\end{proof}

\begin{definition}
We denote by $\T_{\sigma}(R,K)$ the subgroup of $\T_\sigma(R)$ generated by all transvections $I_\sigma+ABC$ where $A\in \Mat_{\sigma\times \tau}(K)$, $B\in \Mat_{\tau\times \tau}(R)$, $C\in \Mat_{\tau\times \sigma}(K)$ for some $\tau\in \Es$, and $CA=0$.
\end{definition}

\begin{lemma}\label{lem:transv}
Suppose that $R$ is a $K$-algebra over a field $K$, and let $\sigma\in \Es$. Then $\T_\sigma(R,K)$ is contained in the subgroupoid of $\Ge(R)$ generated by the sets $\Ee_{\kappa}(R)$ and $\Pe_{\kappa\times\lambda}(K)$ where $\kappa,\lambda\in\Es$.
\end{lemma}

\begin{proof}
Let $A\in \Mat_{\sigma\times \tau}(K)$, $B\in \Mat_{\tau\times \tau}(R)$, $C\in \Mat_{\tau\times \sigma}(K)$ for some $\tau\in \Es$, and $CA=0$. We have to show that the transvection $I_\sigma+ABC$ lies in the subgroupoid $\Ge(R)$ generated by the sets $\Ee_{\kappa}(R)$ and $\Pe_{\kappa\times\lambda}(K)$ where $\kappa,\lambda\in\Es$. Clearly we may assume that $\sigma$ and $\tau$ are ordered. 

Since \(\sigma_i K \tau_j\) (respectively \(\tau_j K \sigma_i\)) is nonzero only when
\(\sigma_i=\tau_j\), we can write
\[
A=\operatorname{diag}(A_1,\dots,A_p),\qquad
C=\operatorname{diag}(C_1,\dots,C_p),
\]
where $A_s\in\Mat_{n^\sigma_s\times n^\tau_s}(K)$ and $C_s\in\Mat_{n^\tau_s\times n^\sigma_s}(K)$ for each $s$. The condition $CA=0$ is equivalent to $C_sA_s=0$ for all $s$. 

For each $s$, choose an $E_s\in\Ee_{n^\sigma_s}(K)$ such that $E_sA_s$ is in row echelon form with $r_s$ pivot elements. Hence  the last $n^\sigma_s-r_s$ rows of $E_sA_s$ are zero. Since $(C_sE_s^{-1})(E_sA_s)=C_sA_s=0$, it follows that the first $r_s$ columns of $C_sE_s^{-1}$ are zero. Let
\[
E=\operatorname{diag}(E_1,\dots,E_p)\in\Ee_\sigma(K).
\]
In order to prove the assertion of the lemma, it suffices to prove that
\[
E(I_\sigma+ABC)E^{-1}=I_\sigma+(EA)B(CE^{-1})
\]
belongs to $\Ee_\sigma(R)$. Clearly 
\[EA=\diag(E_1A_1,\dots,E_pA_p),\qquad CE^{-1}=\diag(C_1E_1^{-1},\dots,C_pE_p^{-1}).\]
Since the last $n^\sigma_s-r_s$ rows of each $E_sA_s$ as well as the first $r_s$ columns of each $C_sE_s^{-1}$ are zero, it follows that $I_\sigma+(EA)B(CE^{-1})$ has the form
\[
\left(
\begin{array}{cc|cc|c|cc}

I_{r_1} & *& 0 & * & \dots & 0&* \\
0 & I_{n^\sigma_1-r_1} & 0 & 0 & \dots & 0&0 \\ \hline
0 & * &I_{r_2}& *  & \dots & 0 & * \\
0 & 0 & 0 &  I_{n^\sigma_2-r_2} & \dots&0 & 0 \\ \hline
\vdots& \vdots & \vdots& \vdots& \ddots & \vdots &\vdots \\\hline
0 & * &0& *  & \dots & I_{r_p} & * \\
0 & 0 & 0  &0&\dots&0&I_{n^\sigma_p-r_p}\\
\end{array}
\right).
\]
But such a matrix clearly lies in $\Ee_\sigma(R)$.
\end{proof}

\section{The general linear groupoid of a free product of rings over a finite product of skew fields}
Recall that given a ring $R_0$, an {\it $R_0$-ring} is a ring $R$ given with a homomorphism
$R_0 \to R$. An $R_0$-ring is called \textit{faithful} if the given homomorphism is an embedding. The $R_0$-rings form a category in which the maps are the ring homomorphisms
that form commutative triangles with the given maps from $R_0$.

In this section:
\begin{itemize}
\medskip
\item $R_0=K^{(1)}\times\dots \times K^{(p)}$ where $K^{(1)},\dots,K^{(p)}$ are skew fields,
\medskip
\item $(R_\lambda)_{\lambda\in\Lambda}$ is a family of faithful $R_0$-rings,
\medskip
\item $R$ denotes the coproduct of the $R_\lambda$ in the
category of $R_0$-rings.
\medskip 
\end{itemize}
\medskip
Note that by \cite[Proposition 2.1]{bergman74} we can identify each $R_\mu$ with its image in $R$. Moreover, if $M=\bigoplus M_\mu\otimes R$ is a standard $R$-module, then we can identify each $M_\mu$ with its image in $M$.

For $1\leq j\leq p$, we denote by $e^{(j)}$ the element of $R_0$ with $1$ in the $j$-th place and $0$'s elsewhere. Clearly the set $\I=\{e^{(1)},\dots,e^{(p)}\}$ is a complete set of orthogonal idempotents in $R_0$ as well as in each $R_\lambda$ and in $R$. As in the previous section, $\Es=\bigcup_{n\geq 1}\I^n$. Moreover, $\Ge(R)$ denotes the general linear groupoid of $R$ with respect to $\I$.

\begin{definition}\label{def:mutransvection}
Let \(\mu\in\Lambda\cup\{0\}\) and $\sigma\in \Es$. If $A\in \Mat_{\sigma\times \tau}(R_\mu)$, $B\in \Mat_{\tau\times \tau}(R)$, $C\in \Mat_{\tau\times \sigma}(R_\mu)$ where $\tau=(e^{(1)},\dots,e^{(p)})\in \Es$, and $CA=0$, then the transvection $I_\sigma+ABC\in \T_\sigma(R)$ is called \textit{$\mu$-based}. We denote by $\T_\sigma(R,\mu)$ the subgroup of $\T_\sigma(R)$ generated by all $\mu$-based transvections.
\end{definition}

\begin{definition}\label{def:G(R,mu)}
Let $\mu\in \Lambda\cup\{0\}$. We denote by $\Ge(R,\mu)$ the subgroupoid of $\Ge(R)$ generated by the sets $\GL_{\sigma\times \tau}(R_{\mu})$ and groups $\T_\sigma(R,\mu)$ where $\sigma,\tau\in\Es$.
\end{definition}

\begin{definition}\label{def:condP}
Suppose that for every $\mu\in\Lambda\cup\{0\}$, $R_\mu$-module $M_\mu$ and pseudo-free $R_\mu$-modules $M'_\mu$ and $M''_\mu$, we have
\begin{align*}
M_\mu \oplus M'_\mu \cong M''_\mu ~\Longrightarrow ~M_\mu \text{ is pseudo-free.}
\end{align*}
Then we say that \textit{Condition (P)} is satisfied.
\end{definition}

\begin{remark}\label{rem:condP}
Note that if for every $\mu\in\Lambda\cup\{0\}$, every finitely generated projective $R_\mu$-module is pseudo-free, then Condition (P) is satisfied.
\end{remark}

In the proof of Theorem \ref{thm:main} below, we follow the proof of \cite[Corollary 2.15]{bergman74}. But since we are assuming that $R_0$ is a \textit{product} of skew fields (not just a single one), the free transfers, where one splits off a summand $xR_{\mu_1}$ from some
$M_{\mu_1}$ and attaches $xR_{\mu_2}$ to $M_{\mu_2}$, must be replaced by \textit{basic} transfers, where one splits off a summand $x(e^{(j)}\!R_{\mu_1})$ from some $M_{\mu_1}$ and attaches $x(e^{(j)}R_{\mu_2})$ to $M_{\mu_2}$ (cf. \cite[\S 9]{bergman74}). Such a basic transfer corresponds to a change of pseudo-basis and hence to the action of a matrix in $\GL_{\pi\times \rho}(R_{\mu_1})$ for some $\pi,\rho\in \Es$.
\begin{theorem}\label{thm:main}
Suppose that Condition (P) is satisfied. Then the general linear groupoid $\Ge(R)$ is generated by the subgroupoids $\Ge(R,\mu)~(\mu\in \Lambda\cup\{0\})$.
\end{theorem}
\begin{proof}
Let $A\in \GL_{\tau\times \sigma}(R)$ be a morphism in $\Ge(R)$. We have to show that $A$ is a composition of morphisms from the subgroupoids $\Ge(R,\mu)~(\mu\in \Lambda\cup\{0\})$. By \S 2.2, the matrix $A$ defines an $R$-module isomorphism $f:M\to N$ where $M=R^{\sigma}$ and $N=R^{\tau}$. 

We consider $M$ and $N$ as standard modules via $M=R^\sigma = R^\sigma_0\otimes_0 R$ and $N=R^\tau = R^\tau_0\otimes_0 R$, and apply \cite[Theorem 2.3]{bergman74} to $f$. Condition (P) ensures that at every step in the transformations of $M$ given by that
theorem, the ``components'' $M_\mu$ will all be pseudo-free.
Hence we can always keep a pseudo-basis of $M$ formed from pseudo-bases of the
current $M_\mu$.

When we perform a basic transfer, splitting off a summand $x(e^{(j)}\!R_{\mu_1})$, where $x\in M_{\mu_1}e^{(j)}$, from some $M_{\mu_1}$, and attaching $x(e^{(j)}R_{\mu_2})$ to $M_{\mu_2}$, let us precede this by changing our pseudo-basis of $M_{\mu_1}$ to one of the form $B' \cup \{x\}$, where $B'$ is a pseudo-basis of the submodule $M'_{\mu_1}$ complementing $x(e^{(j)}\!R_{\mu_1})$. This change of pseudo-basis corresponds to the action of an element of $\GL_{\pi\times \rho}(R_{\mu_1})$, where $\pi$ is the signature of the new pseudo-basis of $M$ and $\rho$ is the signature of the old one. The basic transfer itself then becomes a formal redistribution of the pseudo-basis elements among the $M_\mu$.

When we apply a transvection $\theta$, if the indices $\mu_1$ and $\mu_2$ involved
are distinct, then the matrix representing $\theta$ will be a product of elementary matrices, which are $0$-based transvections, while if $\mu_1=\mu_2$, it will be a ${\mu_1}$-based transvection.

\cite[Theorem 2.3]{bergman74} tells us that after being composed with these automorphisms,
$f$ yields an induced automorphism $f'$. In view of the structure of standard
module we chose for $N$, it follows that $f'$ is presented by a matrix from $\GL_{\tau}(R_{0})$.

It follows from the above, that the matrix $A$ is a finite product of matrices from the sets 
$\GL_{\pi\times \rho}(R_{\mu})$ and $\T_\pi(R,\mu)$ where $\mu\in \Lambda\cup\{0\}$ and $\pi,\rho\in\Es$. Thus $A$ is a composition of morphisms from the subgroupoids $\Ge(R,\mu)~(\mu\in \Lambda\cup\{0\})$.
\end{proof}

\begin{definition}\label{def:G'}
Let $\sigma,\tau\in\Es$. We denote by $\GL_{\sigma\times \tau}(R_\nu;~\nu\in\Lambda\cup\{0\})$ the subset of $\GL_{\sigma\times \tau}(R)$ consisting of all morphisms which are compositions of morphisms from the sets $\GL_{\pi\times \rho}(R_{\mu})~(\mu\in \Lambda\cup\{0\};~\pi,\rho\in\Es)$. Instead of $\GL_{\sigma\times \sigma}(R_\nu;~\nu\in\Lambda\cup\{0\})$ we may write $\GL_{\sigma}(R_\nu;~\nu\in\Lambda\cup\{0\})$
\end{definition}

%\begin{definition}\label{def:hatT}
%Let $\sigma\in \Es$. We denote by $T_\sigma(R)$ the subgroup of $\GL_{\sigma}(R)$ generated by %all matrices of the form $ABA^{-1}$ where $A\in\GL_{\sigma\times \tau}%(R_\nu;~\nu\in\Lambda\cup\{0\})$ and $B\in T_\tau(R,\mu)$ for some $\tau\in \Es$ and $\mu\in \Lambda\cup\{0\}$. 
%\end{definition}

\begin{corollary}\label{cor:main}
Let $\sigma\in\Es$ and suppose that Condition (P) is satisfied. Then every matrix $A\in\GL_\sigma(R)$ can be written as $A=BC$ where $B\in\GL_{\sigma}(R_\nu;~\nu\in\Lambda\cup\{0\})$ and $C\in \T_\sigma(R)$.
\end{corollary}
\begin{proof}
Let $A\in\GL_\sigma(R)$. It follows from Theorem \ref{thm:main} that $A$ can be written as $A=A_1\dots A_n$ where each $A_i$ belongs to $\Gg:=\bigcup_{\pi,\rho\in \Es}\GL_{\pi\times \rho}(R_\nu;~\nu\in\Lambda\cup\{0\})$ or $\T:=\bigcup_{\tau\in \Es}\T_\tau(R)$. Clearly $\T$ is normalised by $\Gg$, i.e. $XYX^{-1}\in \T$ for every composable $X\in \Gg$ and $Y\in \T$. It follows that $A=BC$ for some $B\in\GL_{\pi\times\rho}(R_\nu;~\nu\in\Lambda\cup\{0\})$ and $C\in \T_\tau(R)$ where $\pi,\rho,\tau\in\Es$. Clearly $\pi$, $\rho$ and $\tau$ must be equal to $\sigma$ since $A\in\GL_\sigma(R)$. 
\end{proof}

\section{Applications to Leavitt path algebras of directed graphs}
Throughout this section $K$ denotes a fixed field.
\subsection{Leavitt path algebras of directed graphs}
\begin{definition}
A {\it (directed) graph} is a quadruple $E=(E^0,E^1,r,s)$ where $E^0$ and $E^1$ are sets and $r,s:E^1\rightarrow E^0$ maps. The elements of $E^0$ are called {\it vertices} and the elements of $E^1$ {\it edges}. %If $e$ is an edge, then $s(e)$ is called its {\it source} and $r(e)$ its {\it range}. %If $v$ is a vertex and $e$ an edge, we say that $v$ {\it emits} $e$ if $s(e)=v$, and $v$ {\it receives} $e$ if $r(e)=v$.
A vertex $v$ is called a {\it sink} if $|s^{-1}(v)|=0$, an {\it infinite emitter} if $|s^{-1}(v)|=\infty$, and {\it regular} if $0<|s^{-1}(v)|<\infty$. We denote the subset of $E^0$ consisting of all sinks (resp. regular vertices) by $E^0_{\sink}$ (resp. $E^0_{\reg}$). The graph $E$ is called {\it finite} if $E^0$ and $E^1$ are finite sets.
\end{definition}

\begin{definition}\label{def:lpa}
Let $E$ be a graph. The (associative, but not necessarily unital) $K$-algebra $L(E)$ presented by the generating set $\{v,e,e^*\mid v\in E^0,e\in E^1\}$ and the relations
\begin{enumerate}[(i)]
\item $uv=\delta_{uv}u\quad(u,v\in E^0)$,
\medskip
\item $s(e)e=e=er(e),~r(e)e^*=e^*=e^*s(e)\quad(e\in E^1)$,
\medskip
\item $e^*f= \delta_{ef}r(e)\quad(e,f\in E^1)$ and
\item $\sum_{e\in s^{-1}(v)}ee^*= v\quad(v\in E_{\reg}^0)$
\medskip
\end{enumerate}
is called the {\it Leavitt path algebra} of $E$. 
\end{definition}

Note that if $E$ is finite, then $L(E)$ is unital with $\sum_{v\in E^0}v=1$.

\subsection{Matrix rings over Leavitt path algebras as free products}

Until the end of Section 4, $E$ denotes a finite graph. We write 
\[E^0_{\reg}=\{v_1,\dots,v_p\},~~E^0_{\sink}=\{v_{p+1},\dots v_{q}\}~~\text{ and }~~ s^{-1}(v_i)=\{f_{i,1},\dots,f_{i,m_i}\}\text{ for }1\leq i\leq q\]
(hence $m_i=0$ if $i>p$). We define the $K$-algebra $S:=K^{E^0}$ and denote for any $v\in E^0$ the element of $S$ whose $v$-th component is $1$ and whose other components are $0$ by $\alpha_v$. For any $1\leq i\leq p$ we define the finitely generated projective $S$-modules $P_i:=\alpha_{v_i}S$ and $Q_i:=\bigoplus_{f\in s^{-1}(v_i)}\alpha_{r(f)}S$. Then 
\[L(E)\cong S\big\langle i_1,i_1^{-1}:\overline{P_{1}}\cong \overline{Q_{1}},\dots , i_p,i_p^{-1}:\overline{P_{p}}\cong \overline{Q_{p}}\big\rangle,\]
see \cite[\S 3.2]{lpabook}.

Let $n=|E^0|+|E^1|$. %Note that $n$ is large enough so that
%\[(\bigoplus_{v\in E_{\reg}} P_{v}\oplus Q_{v})\oplus (\bigoplus_{v\in E_{\sink}} vR)\]
%can be embedded as a proper direct summand in the free right $R$-module $R^n$ of rank $n$.
It follows from \cite[Theorem 4.1]{bergman74b} that
 \[\Mat_n(L(E))\cong R\big\langle i_1,i_1^{-1}:\overline{r_n(P_{1})}\cong \overline{r_n(Q_{1})},\dots , i_p,i_p^{-1}:\overline{r_n(P_{p})}\cong \overline{r_n(Q_{p})}\big\rangle,\]
where $R=\Mat_n(S)$ and for any $S$-module $M$, $r_n(M)$ denotes the $R$-module consisting of all row vectors of length $n$ over $M$, on which the matrices in $R=\Mat_n(S)$ act via the usual rules for multiplying a vector by a matrix.

We define the idempotent matrices
\begin{equation}e_{P_i}=\epsilon^{v_i}_{\widehat m_i+1}~~(1\leq i\leq p),\qquad e_{Q_i}=\sum_{j=1}^{m_i}\epsilon^{r(f_{i,j})}_{\widehat m_i+1+j}~~(1\leq i\leq p)\label{E:idemp1}
\end{equation}
and
\begin{equation}
e_i=\epsilon_{\hat m_i+1}^{v_i}~~(p+1\leq i\leq q),\qquad e_0=1-\sum_{i=1}^{p}(e_{P_i}+e_{Q_i})-\sum_{i=p+1}^q e_i \label{E:idemp2}
\end{equation}
where $\widehat m_i=\sum_{l=1}^{i-1}(m_l+1)$, and $\epsilon_{k}^v$ denotes the idempotent matrix in $R$ whose entry at position $(k,k)$ is $\alpha_v$ and whose other entries are $0$. Clearly the matrices
\[e_{P_1},e_{Q_1},e_{P_2},e_{Q_2},\dots, e_{P_p},e_{Q_p},e_{p+1},\dots,e_q, e_0\]
form a complete set of orthogonal idempotents in $R$.

We set $R_0:=K^{p+q+1}$ and view $R$ as an $R_0$-ring$_K$ (see \cite[\S2]{bergman74b}) via the $K$-algebra homomorphism $R_0\to R$ mapping 
\begin{align*}
(x_1,\dots,x_{p+q+1})\mapsto&~e_{P_1}x_1+e_{Q_1}x_2+\dots+e_{P_p}x_{2p-1}+e_{Q_p}x_{2p}\\
&+e_{p+1}x_{2p+1}+\dots+e_{q}x_{p+q}+e_0x_{p+q+1}.
\end{align*}
Note that this homomorphism might not be an isomorphism since $e_0$ can be $0$. We will denote $(1,0,\dots,0)\in R_0$ by $e_{P_1}$, $(0,1,\dots,0)\in R_0$ by $e_{Q_1}$ and so on. 

We leave it to the reader to check that $r_n(P_i)\cong e_{P_i}R$ and $r_n(Q_i)\cong e_{Q_i}R$ for any $1\leq i\leq p$. It follows that the $R$-modules $r_n(P_i),r_n(Q_i)~(1\leq i\leq p)$ are induced by $R_0$-modules:
\begin{align*}
r_n(P_i)&\cong e_{P_i}R\cong P_{0,i} \otimes_{R_0} R, \quad\text{where }P_{0,i}=e_{P_i}R_0,\\[0.2cm]
r_n(Q_i)&\cong e_{Q_i}R\cong Q_{0,i} \otimes_{R_0} R, \quad\text{where }Q_{0,i}=e_{Q_i}R_0.
\end{align*}
By \cite[Theorem 3.4]{bergman74b} this implies that
\begin{align*}
&R\big\langle i_1,i_1^{-1}:\overline{r_n(P_{1})}\cong \overline{r_n(Q_{1})},\dots , i_p,i_p^{-1}:\overline{r_n(P_{p})}\cong \overline{r_n(Q_{p})}\big\rangle\\[0.2cm]
\cong&R\amalg_{R_0}R_0\big\langle i_1,i_1^{-1}:\overline{P_{0,1}}\cong \overline{Q_{0,1}},\dots , i_p,i_p^{-1}:\overline{P_{0,p}}\cong \overline{Q_{0,p}}\big\rangle.
\end{align*}
The factor $R_0\big\langle i_1,i_1^{-1}:\overline{P_{0.1}}\cong \overline{Q_{0,1}},\dots , i_p,i_p^{-1}:\overline{P_{0,p}}\cong \overline{Q_{0,p}}\big\rangle$ is easily seen to be isomorphic to 
\[\underbrace{\Mat_2(K)\times \dots\times\Mat_2(K)}_{p\text{ factors}}\times \underbrace{K\times \dots\times K}_{q-p+1\text{ factors}}\]
made an $R_0$-ring via the homomorphism mapping 
\begin{align*}
(x_1,\dots,x_{p+q+1})\mapsto (e_{11}x_1+e_{22}x_2,\dots,e_{11}x_{2p-1}+e_{22}x_{2p},x_{2p+1},\dots,x_{p+q},x_{p+q+1}),
\end{align*}
cf. \cite[Proof of Theorem 5.2]{bergman74b}. We have shown:

\begin{theorem}\label{thm:lpa1}
Let $E$ be a finite graph with $|E^0_{\reg}|=p$ and $|E^0|=q$. If $n=|E^0|+|E^1|$, then 
\[\Mat_n(L(E))\cong \Mat_n(K^{E^0})\amalg_{K^{p+q+1}}(\Mat_2(K)^p\times K^{q-p+1}).\]
Here $\Mat_n(K^{E^0})$ and $\Mat_2(K)^p\times K^{q-p+1}$ are viewed as $K^{p+q+1}$-rings via the maps
\begin{align*}
(x_1,\dots,x_{p+q+1})\mapsto&~e_{P_1}x_1+e_{Q_1}x_2+\dots+e_{P_p}x_{2p-1}+e_{Q_p}x_{2p}\\
&+e_{p+1}x_{2p+1}+\dots+e_{q}x_{p+q}+e_0x_{p+q+1},
\end{align*}
and 
\begin{align*}
(x_1,\dots,x_{p+q+1})\mapsto (e_{11}x_1+e_{22}x_2,\dots,e_{11}x_{2p-1}+e_{22}x_{2p},x_{2p+1},\dots,x_{p+q},x_{p+q+1}),
\end{align*}
respectively, where the idempotents $e_{P_i}, e_{Q_i}, e_i$ are defined as in (\ref{E:idemp1}) and (\ref{E:idemp2}).
\end{theorem}

\begin{remark}\label{rem:lpa1a}
If $e_0$ equals $0$ (which happens if and only if $E$ has precisely one vertex), then one can remove the last factor $K$ from $K^{p+q+1}$ and also from $\Mat_2(K)^p\times K^{q-p+1}$.
\end{remark}

\begin{remark}\label{rem:lpa1b}
Analysing the proof of Theorem \ref{thm:lpa1} we see that there is an isomorphism 
\[\phi:  \Mat_n(K^{E^0})\amalg_{K^{p+q+1}}(\Mat_2(K)^p\times K^{q-p+1})\longrightarrow\Mat_n(L(E))\]
such that
\begin{align*}
\phi(e_{ij}\alpha_v)&= e_{ij}v &&\text{ for each } e_{ij}\alpha_v \text{ in }\Mat_n(K^{E^0}),\\
\phi(e^{(i)}_{11})&=e_{\hat m_i+1, \widehat m_i+1}v_i &&\text{ for } 1\leq i\leq p,\\
\phi(e^{(i)}_{22})&=\sum_{j=1}^{m_i}e_{\widehat m_i+1+j,\widehat m_i+1+j}r(f_{i,j}) &&\text{ for } 1\leq i\leq p,\\
\phi(e^{(i)}_{12})&=\sum_{j=1}^{m_i}e_{\widehat m_i+1, \widehat m_i+1+j}f_{i,j}&&\text{ for } 1\leq i\leq p,\\
\phi(e^{(i)}_{21})&=\sum_{j=1}^{m_i}e_{ \widehat m_i+1+j,\widehat m_i+1}f_{i,j}^* &&\text{ for } 1\leq i\leq p,\\
\phi(1^{(i)})&=e_{ \widehat m_i+1,\widehat m_i+1}v_i &&\text{ for } p+1\leq i\leq q,\\
\phi(1^{(0)})&=I_n-\sum_{i=1}^p(\phi(e^{(i)}_{11})+\phi(e^{(i)}_{22}))-\sum_{i=p+1}^q\phi(1^{(i)}),&&
\end{align*}
where for $1\leq i\leq p$ and $k,l\in\{1,2\}$,
\[e^{(i)}_{kl}=(0,\dots,0,e_{kl},0,\dots,0)\in \Mat_2(K)^p\times K^{q-p+1}\]
with $e_{kl}$ in the $i$-th position, for $p+1\leq i\leq q$, 
\[1^{(i)}=(0,\dots,0,1,0,\dots,0)\in \Mat_2(K)^p\times K^{q-p+1}\]
with $1$ in the $i$-th position, 
and 
\[1^{(0)}=(0,\dots,0,1)\in \Mat_2(K)^p\times K^{q-p+1}.\]
\end{remark}

\begin{example}\label{ex:lpa1}
Suppose that $E$ is the graph with one vertex $v$ and two edges $e$ and $f$. Note that by \cite[Proposition 1.3.2]{lpabook} the Leavitt path algebra $L(E)$ is isomorphic to the Leavitt algebra $L(1,2)$. By Theorem \ref{thm:lpa1} and Remarks \ref{rem:lpa1a} and \ref{rem:lpa1b} there is an isomorphism  
\[\phi:  \Mat_3(K)\amalg_{K^{2}}\Mat_2(K)\to\Mat_3(L(E))\]
such that
\begin{align*}
\phi(e_{ij}^{(1)})&= e_{ij} \quad\text{for all }1\leq i,j\leq 3,\\
\phi(e^{(2)}_{11})&=e_{11},\\
\phi(e^{(2)}_{22})&=e_{22}+e_{33},\\
\phi(e^{(2)}_{12})&=\begin{pmatrix}
0&e&f\\
0&0&0\\
0&0&0
\end{pmatrix},\\
\phi(e^{(2)}_{21})&=\begin{pmatrix}
0&0&0\\
e^*&0&0\\
f^*&0&0
\end{pmatrix},
\end{align*}
where $e_{ij}^{(1)}$ are the matrix units in $\Mat_3(K)$, $e_{ij}^{(2)}$ are the matrix units in $\Mat_2(K)$ and $e_{ij}$ are the matrix units in $\Mat_3(L(E))$.
\end{example}

\subsection{The functor $\psi:\P(\Mat_n(L(E)))\to\P(L(E))$}

We keep the notation of \S 4.2. Set
\[
R:=\Mat_n(K^{E^0})\amalg_{K^{p+q+1}}\bigl(\Mat_2(K)^p\times K^{q-p+1}\bigr)
\]
and
\(\widetilde{R}:=\Mat_n(L(E))\), and let \(\phi:R\to\widetilde{R}\) be the isomorphism
described in Remark~\ref{rem:lpa1b}.

Let
\[
\mathcal{I}=\{e_{P_i},e_{Q_i},e_j,e_0\mid 1\leq i\leq p,\ p+1\leq j\leq q\}.
\]
Clearly \(\mathcal{J}=\phi(\mathcal{I})\) is a complete set of orthogonal idempotents
in \(\widetilde{R}\), and moreover
\begin{align*}
\phi(e_{P_i}) &= e_{\widehat m_i+1,\widehat m_i+1} v_i && (1\le i\le p),\\[2pt]
\phi(e_{Q_i}) &= \sum_{j=1}^{m_i} e_{\widehat m_i+1+j,\;\widehat m_i+1+j} \, r(f_{i,j}) && (1\le i\le p),\\[2pt]
\phi(e_i) &= e_{\widehat m_i+1,\widehat m_i+1} v_i && (p+1\le i\le q),\\[2pt]
\phi(e_0)
&=
I_n
-
\sum_{i=1}^p
\bigl(\phi(e_{P_i})+\phi(e_{Q_i})\bigr)
-
\sum_{i=p+1}^q
\phi(e_i).
\end{align*}

We denote by \(\P(\widetilde{R})\) the category from Definition~\ref{def:catP}
with respect to \(\mathcal{J}\), and by \(\P(L(E))\) the category from
Definition~\ref{def:catP} with respect to \(E^0=\{v_1,\dots,v_q\}\).
The corresponding general linear groupoids are denoted by
\(\Ge(\widetilde{R})\) and \(\Ge(L(E))\).

\medskip
\noindent\textbf{Atomic decomposition of the idempotents in \(\mathcal{J}\).}
We call the idempotents
\[
\varepsilon(k,v)=e_{kk}v\in \widetilde{R}
\qquad(1\le k\le n,\ v\in E^0)
\]
\textit{atomic}. For \(u\in\mathcal{J}\), let
\(\operatorname{at}(u)\) be the ordered list of atomic idempotents whose
sum is \(u\), ordered first by row index \(k\) and, for equal \(k\), by the
fixed ordering \(v_1,\dots,v_q\) of the vertices. Denote the corresponding
ordered list of vertices by
\[
V(u)=\bigl(\operatorname{vert}(a)\bigr)_{a\in\operatorname{at}(u)},
\]
where \(\operatorname{vert}(\varepsilon(k,v))=v\).

\medskip
\noindent\textbf{Definition of \(\psi\).}
Let
\[
\xi_{\widetilde R}:\P(\widetilde R)\to\P'(\widetilde R)
\quad\text{and}\quad
\xi_{L(E)}:\P(L(E))\to\P'(L(E))
\]
be the canonical isomorphisms from Remark~\ref{rem:xi}.

Let \(\mathcal M\) be the restriction of the Morita functor
\[
-\otimes_{\widetilde R}\widetilde R e_{11}
:
\operatorname{Mod}(\widetilde R)\longrightarrow \operatorname{Mod}(L(E))
\]
to \(\P'(\widetilde R)\). Thus on objects, 
\[
\mathcal M(\widetilde R^\sigma)=\widetilde R^\sigma e_{11}.
\]

Let \(\P''(L(E))\) be the subcategory of
\(\operatorname{Mod}(L(E))\) whose objects are the \(L(E)\)-modules of
the form \(\widetilde R^\sigma e_{11}\), and whose morphisms are all
\(L(E)\)-module homomorphisms between them.

For each \(\sigma\in\Es_{\mathcal J}\), there is a canonical
\(L(E)\)-module isomorphism
\[
\widetilde R^\sigma e_{11}
\;\cong\;
L(E)^{\bar\sigma},
\]
where
\[
\bar\sigma
=
(V(\sigma_1),V(\sigma_2),\dots,V(\sigma_{|\sigma|})).
\]
These isomorphisms induce a functor
\[
\omega:\P''(L(E))\longrightarrow \P'(L(E)).
\]

Finally define
\[
\psi
=
\xi_{L(E)}^{-1}\circ\omega\circ\mathcal M\circ\xi_{\widetilde R}
.
\]
Hence the diagram
\[
\xymatrix@C=1.2cm@R=1.5cm{
\P(\widetilde R)
  \ar[r]^{\xi_{\widetilde R}}
  \ar[d]_{\psi}
& \P'(\widetilde R)
  \ar[dr]^{\mathcal M}
&
\\
\P(L(E))
& \P'(L(E))
  \ar[l]_{\xi_{L(E)}^{-1}}
& \mathcal \P''(L(E))
  \ar[l]_{\omega}
}
\]
commutes.

\begin{remark}\label{rem:additive}
Note that \(\psi\) preserves sums of
matrices.
\end{remark}

\begin{remark}\label{rem:psi}
Concretely, \(\psi\) does the following when applied to objects or morphisms.

\smallskip
\noindent\textit{Objects.}
For \(\sigma=(u_1,\dots,u_m)\in\Es_{\mathcal J}\),
\[
\psi(\sigma)
=
(V(u_1),V(u_2),\dots,V(u_m)),
\]
where the right-hand side is understood as the concatenation of the
lists \(V(u_i)\).

\smallskip
\noindent\textit{Morphisms.}
Let \(A=(A_{ij})\in\Mat_{\tau\times\sigma}(\widetilde R)\). Then \(A_{ij}\in \tau_i\widetilde R\sigma_j\) for all $i$ and $j$. Since the idempotents \(\tau_i\) and \(\sigma_j\) decompose into atomic idempotents, each entry \(A_{ij}\) of $A$ splits uniquely as a sum
\[
A_{ij}
=
\sum_{\substack{a\in\operatorname{at}(\tau_i),\\ b\in\operatorname{at}(\sigma_j)}}
A_{ij}^{a,b}
\]
where each $A_{ij}^{a,b}\in a\widetilde R b$. Let $B$ be the matrix whose rows are indexed by the elements of the ordered list $(\operatorname{at}(\tau_1),\dots,\operatorname{at}(\tau_{|\tau|}))$, whose columns are indexed by the elements of the ordered list $(\operatorname{at}(\sigma_1),\dots,\operatorname{at}(\sigma_{|\sigma|}))$, and whose entry corresponding to
\(a\in\operatorname{at}(\tau_i)\) and \(b\in\operatorname{at}(\sigma_j)\)
is the only nonzero entry of the matrix $A_{ij}^{a,b}$ (respectively $0$, if $A_{ij}^{a,b}$ has no nonzero entries). Then \(B=\psi(A)\).
\end{remark}

\begin{example}\label{ex:psi}
Let again $E$ be the graph with one vertex and two edges. Let 
\[A=\begin{pmatrix}
\begin{pmatrix}
a_{11}&0&0\\
0&0&0\\
0&0&0
\end{pmatrix}&\begin{pmatrix}
0&a_{12}&a_{13}\\
0&0&0\\
0&0&0
\end{pmatrix}\\
\begin{pmatrix}
0&0&0\\
a_{21}&0&0\\
a_{31}&0&0
\end{pmatrix}&\begin{pmatrix}
0&0&0\\
0&a_{22}&a_{23}\\
0&a_{32}&a_{33}
\end{pmatrix}
\end{pmatrix}
\in \Mat_{\sigma\times\sigma}(\Mat_3(L(E))),\]
where $a_{ij}\in L(E)$ for all $i,j$ and $\sigma= (e_{11},e_{22}+e_{33})$. Then 
\[\psi(A)=\begin{pmatrix}
a_{11}&a_{12}&a_{13}\\
a_{21}&a_{22}&a_{23}\\
a_{31}&a_{32}&a_{33}
\end{pmatrix}.\]
\end{example}

Clearly $\psi$ restricts to a functor $\Ge(\Mat_n(L(E)))\longrightarrow \Ge(L(E))$. We denote this restriction by the same letter $\psi$.

\begin{lemma}\label{lem:psisurj}
The functor
\[
\psi:
\Ge(\Mat_n(L(E)))\longrightarrow \Ge(L(E))
\]
is surjective on objects and on
morphisms.
\end{lemma}

\begin{proof}
$~$\\
\noindent\textit{Surjective on objects.}
Let \(\alpha=(v_{j_1},\dots,v_{j_a})\in\Es_{E^0}\). For each vertex
\(v_j\in E^0\), define the idempotent
\[
u(v_j)=
\begin{cases}
\phi(e_{P_j}) & \text{if }v_j\text{ is regular }(1\le j\le p),\\[2pt]
\phi(e_j)    & \text{if }v_j\text{ is a sink }(p+1\le j\le q).
\end{cases}
\]
Then \(u(v_j)\) is a single atomic idempotent and
\(V(u(v_j))=(v_j)\). Set
\[
\sigma
=
\bigl(u(v_{j_1}),u(v_{j_2}),\dots,u(v_{j_a})\bigr)
\in \Es_{\mathcal J}.
\]
Then
\[
\psi(\sigma)
=
(V(u(v_{j_1})),V(u(v_{j_2})),\dots,V(u(v_{j_a})))
=
(v_{j_1},v_{j_2},\dots,v_{j_a})
=
\alpha.
\]

\smallskip
\noindent\textit{Surjective on morphisms.}
Let \(\alpha=(v_{j_1},\dots,v_{j_a})\) and
\(\beta=(v_{j'_1},\dots,v_{j'_b})\) be objects in \(\Ge(L(E))\), and
let
\[
X\in\GL_{\beta\times\alpha}(L(E)).
\]
Define
\[
\sigma
=
\bigl(u(v_{j_1}),\dots,u(v_{j_a})\bigr),
\qquad
\tau
=
\bigl(u(v_{j'_1}),\dots,u(v_{j'_b})\bigr).
\]
Then \(\psi(\sigma)=\alpha\) and \(\psi(\tau)=\beta\). For each
\(1\le j\le q\), write
\[
u(v_j)=e_{k_j,k_j}v_j,
\qquad
k_j=\widehat m_j+1.
\]
Define a matrix \(A\in\Mat_{\tau\times\sigma}(\widetilde R)\) by
\[
A_{rs}=e_{k_{j'_r},\,k_{j_s}}\,X_{rs}
\qquad(1\le r\le b,\ 1\le s\le a).
\]
If \(Y=X^{-1}\),
define \(B\in\Mat_{\sigma\times\tau}(\widetilde R)\) by
\[
B_{sr}=e_{k_{j_s},\,k_{j'_r}}\,Y_{sr}\qquad ( 1\le s\le a,\ 1\le r\le b).
\]
Then \(AB=I_\tau\) and \(BA=I_\sigma\), so
\(A\in\GL_{\tau\times\sigma}(\widetilde R)\). By Remark~\ref{rem:psi},
\[
\psi(A)_{rs}=X_{rs}
\]
for all \(r,s\). Thus \(\psi(A)=X\). 
\end{proof}

\subsection{A generating set for $\Ge(L(E))$}
Recall that $E$ denotes a finite graph and
\[
E^0_{\reg}=\{v_1,\dots,v_p\},\qquad
E^0_{\sink}=\{v_{p+1},\dots,v_q\},\qquad
s^{-1}(v_i)=\{f_{i,1},\dots,f_{i,m_i}\}\quad (1\leq i\leq q).
\]
We continue to write $\Ge(L(E))$ for the general linear groupoid of
$L(E)$ with respect to $E^0=\{v_1,\dots,v_q\}$.

\begin{definition}\label{def:F}
Let $\sigma,\tau\in \Es_{E^0}$. We denote by $\Ef_{\sigma\times\tau}(L(E))$ the set of
all block-diagonal matrices
$A=\operatorname{diag}(A_1,\dots,A_\ell)\in\GL_{\sigma\times\tau}(L(E))$
such that each block $A_j$ is either of the
form
\begin{enumerate}[(i)]
\item $(v)$ where $v\in E^0$, or
\item $(f_{i,1},\dots,f_{i,m_i})$ where
$1\leq i\leq p$, or
\item $(f_{i,1},\dots,f_{i,m_i})^*
      =
      \begin{pmatrix}
       f_{i,1}^*\\
      \vdots\\
       f_{i,m_i}^*
      \end{pmatrix}$
      where $1\leq i\leq p$.
\end{enumerate}
\end{definition}

\begin{example}
Suppose again that $E$ is the graph with one vertex $v$ and two edges $e$ and $f$. Let
\[
\sigma=(v,v,v,v,v),\qquad \tau=(v,v,v,v,v,v).
\]
Then
\[
A=\operatorname{diag}
\bigl(
(e\ f),\ (v),\ (e\ f)^*,\ (e\ f)
\bigr)
=
\begin{pmatrix}
e & f & 0 & 0 & 0 & 0\\
0 & 0 & v & 0 & 0 & 0\\
0 & 0 & 0 & e^* & 0 & 0\\
0 & 0 & 0 & f^* & 0 & 0\\
0 & 0 & 0 & 0 & e & f
\end{pmatrix}
\in \Ef_{\sigma\times\tau}(L(E)).
\]
Its inverse is the matrix 
\[
A^{-1}
=
A^*
=
\operatorname{diag}
\bigl(
(e\ f)^*,\ (v),\ (e\ f),\ (e\ f)^*
\bigr)
=
\begin{pmatrix}
e^* & 0 & 0 & 0 & 0\\
f^* & 0 & 0 & 0 & 0\\
0 & v & 0 & 0 & 0\\
0 & 0 & e & f & 0\\
0 & 0 & 0 & 0 & e^*\\
0 & 0 & 0 & 0 & f^*
\end{pmatrix}
\in\Ef_{\tau\times\sigma}(L(E))
.\]
\end{example}

\begin{theorem}\label{thm:lpa2}
The general linear groupoid $\Ge(L(E))$ is generated by the sets
\[
\Ee_{\sigma}(L(E)),~\Ef_{\sigma\times\tau}(L(E)),~\Pe_{\sigma\times\tau}(K)
\quad (\sigma,\tau\in\Es_{E^0}).
\]
\end{theorem}
\begin{proof}
By Theorem~\ref{thm:lpa1} there is an isomorphism
\[
\phi : R\to \Mat_n(L(E)),
\qquad
R = \Mat_n(K^{E^0})\amalg_{K^{p+q+1}}\bigl(\Mat_2(K)^p\times K^{q-p+1}\bigr).
\]
Write $R_0=K^{p+q+1}$, $R_1=\Mat_n(K^{E^0})$ and $R_2=\Mat_2(K)^p\times K^{q-p+1}$. We will consider only the case that $e_0\neq 0$. The case $e_0=0$ can be handled similarly, but one should delete the last factor \(K\) from \(R_0\) and the last
factor \(K\) from \(R_2\) (see Remark \ref{rem:lpa1a}). Both $R_1$ and $R_2$ are faithful $R_0$-rings, and $R$ is their coproduct in the category of $R_0$-rings.  

Recall that the $\V$-monoid $\V(S)$ of a ring $S$ is the set of all isomorphism classes of finitely generated projective modules, which becomes an abelian monoid by defining $[P]+[Q]=[P\oplus Q]$.

The monoid $\V(R_0)$ is isomorphic to $\N_0^r$. In $\N_0^r$, the element $\alpha_i$, whose $i$-th component is $1$ and whose other components are $0$, represents the isomorphism class of $u_iR_0$. 
 
The monoid $\V(R_1)$ is isomorphic to $\N_0^q$. In $\N_0^q$, the element $\alpha_j$ represents the isomorphism class of $e_{11}\alpha_{v_j}R_1$.  

The monoid $\V(R_2)$ is isomorphic to $\N_0^{q+1}$. If $1\leq i\leq p$, then $\alpha_i\in\N_0^{q+1}$ represents the isomorphism class of $e_{11}^{(i)}R_2$ where 
\[e^{(i)}_{11}=(0,\dots,0,e_{11},0,\dots,0)\]
with $e_{11}$ in the $i$-th position. If $p+1\leq i\leq q+1$, then $\alpha_i\in\N_0^{q+1}$ represents the isomorphism class of $1^{(i)}R_2$ where
\[1^{(i)}=(0,\dots,0,1,0,\dots,0)\]
with $1$ in the $i$-th position.

Considering the $V$-monoids of $R_0$, $R_1$ and $R_2$ we see that every finitely generated projective module over one of these rings is pseudo‑free. Hence Condition (P) is satisfied by Remark \ref{rem:condP}. Theorem~\ref{thm:main} therefore tells us that the general linear
groupoid $\Ge(R)$ with respect to
\[\mathcal{I}=\{e_{P_i},e_{Q_i},e_j,e_0\mid 1\leq i\leq p,\ p+1\leq j\leq q\}=\{u_1,\dots,u_r\}\]
is generated by the subgroupoids $\Ge(R,\mu)\;(\mu=0,1,2)$.
Recall that for $\mu\in\{0,1,2\}$, $\Ge(R,\mu)$ denotes the subgroupoid of $\Ge(R)$ generated by
the sets $\GL_{\sigma\times\tau}(R_\mu)$ and 
$\T_\sigma(R,\mu)$ where $\sigma,\tau\in \Es_\I$.

Clearly the isomorphism $\phi:R\to\Mat_n(L(E))$ in Remark \ref{rem:lpa1b} induces an isomorphism
\[
\phi:\Ge(R)\to\Ge(\Mat_n(L(E))),
\]
where $\Ge(\Mat_n(L(E)))$ denotes the general linear groupoid of
$\Mat_n(L(E))$ with respect to $\mathcal{J}=\phi(\mathcal{I})$.
Let
\[
\psi:\Ge(\Mat_n(L(E)))\to\Ge(L(E))
\]
be the functor defined in \S 4.3. Since $\psi$ is surjective by
Lemma~\ref{lem:psisurj}, the composition
\[
\psi\circ\phi : \Ge(R)\to\Ge(L(E))
\]
is surjective. It therefore suffices to show that all generators of
$\Ge(R)$ are mapped by $\psi\circ\phi$ into the subgroupoid
$\H\subseteq\Ge(L(E))$ generated by the sets
$\Ee_{\sigma}(L(E))$, $\Ef_{\sigma\times\tau}(L(E))$ and $\Pe_{\sigma\times\tau}(K)
$ where $\sigma,\tau\in\Es_{E^0}$. 

If $\sigma$ is an object in $\Ge(R)$ and $A$ is a morphism in $\Ge(R)$, then we might denote their images in $\Ge(L(E))$ under $\psi\circ\phi$ by $\bar\sigma$ and $\bar A$, respectively. 
 
\bigskip\noindent 
\textbf{Part I}\\
First we will show that $\psi(\phi(\GL_{\sigma\times\tau}(R_\mu)))\subseteq \H$ for every $\mu\in\{0,1,2\}$ and $\sigma,\tau\in \Es_\I$. 

\medskip\noindent
\textit{1. Matrices from $\GL_{\sigma\times\tau}(R_0)$.}
Let $\sigma,\tau\in\Es_\I$ and $A\in\GL_{\sigma\times\tau}(R_0)$. Let $\sigma_0\in \Es_\I$ be the ordered element that can be obtained from $\sigma$ by permuting components. Similarly, let $\tau_0\in \Es_\I$ be the ordered element that can be obtained from $\tau$ by permuting components (see Definition \ref{def:ordered}). Choose permutation matrices $P\in\Pe_{\sigma_0\times\sigma}(K)$ and $Q\in\Pe_{\tau\times\tau_0}(K)$. Then $\hat A:=PAQ\in \GL_{\sigma_0\times\tau_0}(R_0)$. This shows that we are reduced to the case that $\sigma$ and $\tau$ are ordered, since $\bar P\in \Pe_{\bar\sigma_0\times\bar\sigma}(K)$ and $\bar Q\in \Pe_{\bar\tau\times\bar\tau_0}(K)$.

So assume $\sigma,\tau\in\Es_\I$ are ordered and $A\in\GL_{\sigma\times\tau}(R_0)$. The matrix $A$ defines an isomorphism between the finitely generated projective $R_0$-modules
\[R_0^\sigma\cong \bigoplus_{i=1}^{r}(u_iR_0)^{n^\sigma_i}\quad\text{and}\quad 
R_0^\tau\cong \bigoplus_{i=1}^{r}(u_iR_0)^{n^\tau_i}.\] 
Considering the $\V$-monoid of $R_0$ we see that $n^\sigma_i=n^\tau_i$ for every
$1\leq i\leq r$. It follows that $\sigma=\tau$. Hence $\bar A\in \GL_{\bar\sigma}(K)$. It follows from Lemma~\ref{lem:GE} that $\bar A\in \H$.

\medskip\noindent
\textit{2. Matrices from $\GL_{\sigma\times\tau}(R_1)$.}
Let $\sigma,\tau\in\Es_\I$ and $A\in\GL_{\sigma\times\tau}(R_1)$. Then $A$
defines an isomorphism between the finitely generated projective
$R_1$-modules
\[
R_1^\sigma
\;\cong\;
\bigoplus_{j=1}^{|\sigma|}
\bigoplus_{v\in V(\phi(\sigma_j))} e_{11}\alpha_v R_1
\qquad\text{and}\qquad
R_1^\tau
\;\cong\;
\bigoplus_{j=1}^{|\tau|}
\bigoplus_{v\in V(\phi(\tau_j))} e_{11}\alpha_v R_1.
\]
Considering the \(\V\)-monoid of \(R_1\) we see that for every vertex
\(v\in E^0\), the number of direct summands of the form 
\(e_{11}\alpha_{v}R_1\) on the left hand side must be equal to the number of such summands on the right hand side.

Clearly $\bar\sigma=(V(\phi(\sigma_1)),\dots,V(\phi(\sigma_{|\sigma|})))$ and $\bar\tau=(V(\phi(\tau_1)),\dots,$ $V(\phi(\tau_{|\tau|})))$. It follows from the previous paragraph that there is a permutation $\pi$ such that $\bar\sigma_i=\bar\tau_{\pi(i)}$ for all $1\leq i\leq |\bar\sigma|=|\bar\tau|$. Clearly $\bar A\in \GL_{\bar\sigma\times \bar\tau}(K)$. Choose a matrix $P\in \Pe_{\bar\tau\times\bar\sigma}(K)$. Then $\bar AP\in \GL_{\bar\sigma}(K)$. It follows from Lemma~\ref{lem:GE} that $\bar AP\in \H$. Thus $\bar A\in \H$.

\medskip\noindent
\textit{3. Matrices from $\GL_{\sigma\times\tau}(R_2)$.}
Let $\sigma,\tau\in \Es_\I$ and $A\in\GL_{\sigma\times\tau}(R_2)$. As in 1., we can assume that $\sigma$ and $\tau$ are ordered. The matrix $A$ defines an isomorphism between the finitely generated projective $R_2$-modules
\[\bigoplus_{i=1}^r(u_iR_2)^{n^\sigma_i}\cong\bigoplus_{i=1}^p(e^{(i)}_{11}R_2)^{n^\sigma_{2i-1}+n^\sigma_{2i}} \oplus \bigoplus_{i=p+1}^{q+1}(1^{(i)}R_2)^{n^\sigma_{i+p}}\]
and
\[\bigoplus_{i=1}^r(u_iR_2)^{n^\tau_i}\cong\bigoplus_{i=1}^p(e^{(i)}_{11}R_2)^{n^\tau_{2i-1}+n^\tau_{2i}} \oplus \bigoplus_{i=p+1}^{q+1}(1^{(i)}R_2)^{n^\tau_{i+p}}.\]
Considering the $\V$-monoid of $R_2$ we see that $n^\sigma_{2i-1}+n^\sigma_{2i}=n^\tau_{2i-1}+n^\tau_{2i}$ for every $1\leq i\leq p$, and $n^\sigma_i=n^\tau_i$ for every $2p+1\leq i\leq p+q+1$. Hence $|\sigma|=|\tau|$.

Set $n^{(i)}:=n^\sigma_{2i-1}+n^\sigma_{2i}$ for $1\leq i\leq p$ and $n^{(i)}:=n^\sigma_{i+p}$ for $p+1\leq i\leq q+1$. Clearly $A=\diag(A^{(1)},\dots,A^{(q+1)})$ where each $A^{(i)}$ is a $n^{(i)}\times n^{(i)}$-matrix. Using the Gauss algorithm, we can find a matrix $E\in \Ee_\sigma(R_2)$ such that $EA$ is a diagonal matrix in $\GL_{\sigma\times\tau}(R_2)$. Clearly each diagonal entry of $EA$ is either of the form 
\begin{itemize}
\item $ke^{(i)}_{lm}$ where $k\in K^\times$, $1\leq i\leq p$ and $l,m\in\{1,2\}$, or
\item \(k 1^{(i)}\) where \(k\in K^\times\) and \(p+1\leq i\leq q+1\).
\end{itemize}
It follows that $\overline{EA}\in  \De_{\bar\sigma}(K)\Ef_{\bar\sigma\times\bar\tau}(L(E))\subseteq\H$. Since $\bar E\in \Ee_{\bar \sigma}(L(E))\subseteq \H$, this implies that $\bar A\in \H$.

\bigskip\noindent 
\textbf{Part II}\\
Next we will show that $\psi(\phi(\T_{\sigma}(R_\mu)))\subseteq \H$ for every $\mu\in\{0,1,2\}$ and $\sigma\in \Es_\I$.

\medskip\noindent
\textit{4. Transvections from \(\T_\sigma(R,0)\) and \(\T_\sigma(R,1)\).}\\
Let $\mu\in \{0,1\}$, $\sigma\in\Es_\I$ and \(T=I_\sigma+ABC\in \T_\sigma(R,\mu)\), where
$A$ and $C$ have entries in \(R_\mu\) and \(CA=0\).
Since $\phi$ and $\psi$ are additive, we have
\[
\bar T
=
I_{\bar\sigma}+\bar A\bar B\bar C,
\]
where \(\bar A,\bar C\) have entries in \(K\), \(\bar B\) has entries in
\(L(E)\), and \(\bar C\bar A=0\). It follows from Lemma~\ref{lem:transv} that \(\bar T\in \H\).

\medskip\noindent
\textit{5. Transvections from \(\T_\sigma(R,2)\).}
Let $\sigma\in \Es_\I$ and $T\in \T_\sigma(R,2)$. Then there are matrices
$A\in\Mat_{\sigma\times\tau}(R_2)$, $B\in\Mat_{\tau\times\tau}(R)$ and
$C\in\Mat_{\tau\times\sigma}(R_2)$ where $\tau=(u_1,\dots,u_r)$ such
that \(CA=0\) and
\[
T=I_\sigma+ABC.
\]
As in Part~I, we may assume that \(\sigma\) is ordered (note that
$\tau$ is already ordered). Set
\[
n^{(i)}:=n^\sigma_{2i-1}+n^\sigma_{2i}
\quad(1\leq i\leq p),
\qquad
n^{(i)}:=n^\sigma_{i+p}
\quad(p+1\leq i\leq q+1).
\]
Clearly
\[
A=\operatorname{diag}(A^{(1)},\dots,A^{(q+1)}),
\]
where each $A^{(i)}$ is an $n^{(i)}\times 2$-matrix if
$1\leq i\leq p$, respectively an $n^{(i)}\times 1$-matrix if
$p+1\leq i\leq q+1$. Similarly,
\[
C=\operatorname{diag}(C^{(1)},\dots,C^{(q+1)}),
\]
where each $C^{(i)}$ is a $2\times n^{(i)}$-matrix if
$1\leq i\leq p$, respectively a $1\times n^{(i)}$-matrix if
$p+1\leq i\leq q+1$. The condition $CA=0$ is equivalent to
\[
C^{(i)}A^{(i)}=0
\qquad
(1\le i\le q+1).
\]

We proceed as in the proof of Lemma~\ref{lem:transv}. We can find a
matrix
\[
E=\operatorname{diag}(E^{(1)},\dots,E^{(q+1)})\in\Ee_\sigma(R_2),
\]
where $E^{(i)}$ is an $n^{(i)}\times n^{(i)}$-matrix for each $i$,
such that in
\[
EA=\operatorname{diag}(E^{(1)}A^{(1)},\dots,E^{(q+1)}A^{(q+1)})
\]
each matrix $E^{(i)}A^{(i)}$ is in row echelon form with $r_i$
pivot elements. Hence the last $n^{(i)}-r_i$ rows of each
$E^{(i)}A^{(i)}$ are zero. Since
\[
(C^{(i)}{E^{(i)}}^{-1})(E^{(i)}A^{(i)})
=
C^{(i)}A^{(i)}
=
0,
\]
it follows that the first $r_i$ columns of each
$C^{(i)}{E^{(i)}}^{-1}$ are zero.

As in the proof of Lemma~\ref{lem:transv}, it follows that
\[
E(I_\sigma+ABC)E^{-1}\in \Ee_\sigma(R).
\]
Since $\psi\circ\phi$ maps elementary groups into elementary groups,
this implies that $\bar T\in\H$.
\end{proof}

\begin{remark}\label{rem:one-vertex}
It is easy to see, that if the graph \(E\) has exactly one vertex, then every permutation matrix in $\Ge(L(E))$ can be written as a product of elementary matrices and one invertible diagonal matrix with entries in $K$. Consequently, for one-vertex graphs, Theorem~\ref{thm:lpa2} remains true if one replaces $\Pe_{\sigma\times\tau}(K)$ by \(\De_\sigma(K)\).
\end{remark}

\subsection{Analysis of the generating set}
Until the end of Section 4 we denote $L(E)$ by $L$.

\subsubsection{Definitions}

\begin{definition}\label{def:F+F-}
Let $\sigma,\tau\in \Es_{E^0}$. We denote by $\Ef^+_{\sigma\times\tau}(L)$
(respectively $\Ef^-_{\sigma\times\tau}(L)$) the subset of
$\Ef_{\sigma\times\tau}(L)$ consisting of all block-diagonal matrices
$A=\operatorname{diag}(A_1,\dots,A_\ell)\in\GL_{\sigma\times\tau}(L)$
such that each block $A_i$ is of the form (i) or (ii)
(respectively (i) or (iii)) in Definition~\ref{def:F}.
\end{definition}

\begin{remark}
Note that if $A=\operatorname{diag}(A_1,\dots,A_\ell)\in \Ef^+_{\sigma\times\tau}(L)$, then \[A^{-1}=A^*=\operatorname{diag}(A_1^*,\dots,A_\ell^*)\in \Ef^-_{\tau\times\sigma}(L).\]
Similarly, if $B\in \Ef^-_{\sigma\times\tau}(L)$, then $B^{-1}=B^*\in\Ef^+_{\tau\times\sigma}(L)$.
\end{remark}

\begin{definition}\label{def:fatALL}
We denote by 
\begin{itemize}
\item \(\EE(L)\) the subgroupoid of $\Ge(L)$ generated by all sets \(\Ee_\sigma(L)\) where \(\sigma\in\Es\),
\item \(\F(L)\) the subgroupoid of $\Ge(L)$ generated by all sets \(\Ef_{\sigma\times\tau}(L)\) where \(\sigma,\tau\in\Es\),
\item $\F^+(L)$ (resp. $\F^-(L)$) the subcategory of
$\Ge(L)$ generated by the sets $\Ef^+_{\sigma\times\tau}(L)$ (resp. $\Ef^-_{\sigma\times\tau}(L)$) where $\sigma,\tau\in\Es$,
\item \(\P(L)\) the subgroupoid of $\Ge(L)$ generated by all sets \(\Pe_{\sigma\times\tau}(K)\) where \(\sigma,\tau\in\Es\),
\item $\F\P(L)$ the subgroupoid of $\Ge(L)$ generated by all sets
$\Ef_{\sigma\times\tau}(L)$ and $\Pe_{\sigma\times\tau}(K)$
where $\sigma,\tau\in\Es$,
\item $\widehat{\EE}(L)$ the subgroupoid of $\Ge(L)$ generated by all matrices of the form $AEA^{-1}$ where $A\in \F\P(L)$ and $E\in \EE(L)$ are composable. 
\end{itemize}
For $\mathbf{C}\in\{\EE,\F,\F^+,\F^-,\P,\F\P,\widehat\EE\}$ and $\sigma,\tau\in \Es$ we set $\mathbf{C}_{\sigma\times\tau}(L):=\mathbf{C}(L)\cap \GL_{\sigma\times \tau}(L)$ and $\mathbf{C}_{\sigma}(L):=\mathbf{C}_{\sigma\times\sigma}(L)$.
\end{definition}

\subsubsection{An obvious corollary}
\begin{corollary}\label{cor:lpa2}
Let $\sigma\in \Es$. Then every matrix $A\in \GL_\sigma(L)$ can be written as $A=BC$ where $B\in \widehat{\EE}_\sigma(L)$ and $C\in\F \P_{\sigma}(L)$.
\end{corollary}
\begin{proof}
By Theorem \ref{thm:lpa2} a matrix $A\in \GL_\sigma(L)$ can be written as a composition of matrices from $\widehat{\EE}(L)$ and $\F\P(L)$. Clearly $\widehat{\EE}(L)$ is normalised by $\F\P(L)$, i.e. if $X\in \F\P(L)$ and $Y\in \widehat{\EE}(L)$ are composable, then $XYX^{-1}\in  \widehat{\EE}(L)$. It follows that $A$ can be written as $A=BC$ where $B\in \widehat{\EE}_\sigma(L)$ and $C\in \F\P_{\sigma}(L)$.
\end{proof}

\subsubsection{$\F^+\F^-$ factorisation of $\F(L)$}

Our first goal is to show that any $F\in \F(L)$ can be written as $F=F^+F^-$ where
$F^+\in \F^+(L)$ and
$F^-\in \F^-(L)$. In order to prove this we need two lemmas.

\begin{lemma}\label{lem:F1}
Let $\sigma,\tau\in \Es$ and
$A\in \Ef_{\sigma\times\tau}(L)$. Then there is a $\pi\in \Es$
and matrices
$B\in \Ef^+_{\sigma\times\pi}(L)$ and
$C\in \Ef^-_{\pi\times\tau}(L)$
such that $A=BC$.
\end{lemma}
\begin{proof}
Write
$A=\operatorname{diag}(A_1,\dots,A_\ell)\in\GL_{\sigma\times\tau}(L)$
where each block $A_i$ is of the form (i), (ii) or (iii) in
Definition~\ref{def:F}. Define the matrix
$B=\operatorname{diag}(B_1,\dots,B_\ell)$ by
\[
B_j=\begin{cases}
(v), & \text{if }A_j=(v)\text{ where }v\in E^0,\\[2pt]
(f_{i,1},\dots,f_{i,m_i}), & \text{if }
A_j=(f_{i,1},\dots,f_{i,m_i})\text{ where }1\le i\le p,\\[2pt]
\operatorname{diag}\bigl(r(f_{i,1}),\dots,r(f_{i,m_i})\bigr), & \text{if }
A_j=(f_{i,1},\dots,f_{i,m_i})^*\text{ where }1\le i\le p.
\end{cases}
\]
Then clearly
$B\in \Ef^+_{\sigma\times\pi}(L)$ for some $\pi\in\Es$.
Define the matrix
$C=\operatorname{diag}(C_1,\dots,C_\ell)$ by
\[
C_j=\begin{cases}
(v), & \text{if }A_j=(v)\text{ where }v\in E^0,\\[2pt]
\operatorname{diag}\bigl(r(f_{i,1}),\dots,r(f_{i,m_i})\bigr), & \text{if }
A_j=(f_{i,1},\dots,f_{i,m_i})\text{ where }1\le i\le p,\\[2pt]
(f_{i,1},\dots,f_{i,m_i})^*, & \text{if }
A_j=(f_{i,1},\dots,f_{i,m_i})^*\text{ where }1\le i\le p.
\end{cases}
\]
Then clearly
$C\in \Ef^-_{\pi\times\tau}(L)$, and moreover
\[
A=\operatorname{diag}(B_1C_1,\dots,B_\ell C_\ell)=BC.
\]
\end{proof}

\begin{lemma}\label{lem:F2}
Let
$A\in \Ef^-_{\sigma\times\tau}(L)$ and
$B\in \Ef^+_{\tau\times\pi}(L)$, where
$\sigma,\tau,\pi\in \Es$. Then
$AB\in \Ef_{\sigma\times\pi}(L)$.
\end{lemma}
\begin{proof}
Write
$A=\operatorname{diag}(A_1,\dots,A_\ell)\in\GL_{\sigma\times\tau}(L)$,
where each block $A_i$ is of the form (i) or (iii) in
Definition~\ref{def:F}. Similarly write
$B=\operatorname{diag}(B_1,\dots,B_m)\in\GL_{\tau\times\pi}(L)$,
where each block $B_i$ is of the form (i) or (ii) in
Definition~\ref{def:F}. Then clearly $\ell=m=|\tau|$. It follows that
\[
AB=\operatorname{diag}(A_1B_1,\dots,A_\ell B_\ell).
\]
One checks easily that each of the blocks $A_iB_i$ is of the
form (i), (ii) or (iii) in Definition~\ref{def:F}, or of the form
$\operatorname{diag}(v,\dots,v)$ where $v\in E^0$. Thus
\[
AB\in \Ef_{\sigma\times\pi}(L).
\]
\end{proof}

\begin{proposition}\label{prop:F+F-}
Let $F\in \F(L)$. Then there is an
$F^+\in \F^+(L)$ and an
$F^-\in \F^-(L)$
such that $F=F^+F^-$.
\end{proposition}
\begin{proof}
By Lemma~\ref{lem:F1} we can write $F=F_1\cdots F_n$, where each
$F_i$ belongs to $\Ef^+_{\sigma\times\tau}(L)$ or to
$\Ef^-_{\sigma\times\tau}(L)$ for some $\sigma,\tau\in\Es$.
It follows from Lemmas~\ref{lem:F1} and~\ref{lem:F2} that if
$F_i\in \Ef^-_{\sigma\times\tau}(L)$ and
$F_{i+1}\in \Ef^+_{\tau\times\pi}(L)$ for some
$1\le i\le n-1$ and $\sigma,\tau,\pi\in\Es$, then
\[
F_iF_{i+1}=F'_iF'_{i+1},
\]
where
$F'_i\in \Ef^+_{\sigma\times\tau'}(L)$ and
$F'_{i+1}\in \Ef^-_{\tau'\times\pi}(L)$
for some $\tau'\in\Es$. After applying a finite number of such
``transpositions'' we obtain the desired factorisation of $F$.
\end{proof}

\subsubsection{$\F^+\P\F^-$ factorisation of $\F\P(L)$}
Next we want to prove that any $A\in \F\P(L)$ can be written as $A=F^+PF^-$ where $F^+\in \F^+(L)$, $P\in \P(L)$ and $F^-\in \F^-(L)$. We need three lemmas.

\begin{lemma}\label{lem:commute1}
Let $P\in \Pe_{\sigma\times\tau}(K)$ and $F\in \Ef^+_{\tau\times\rho}(L)$. Then there is an $\omega\in\Es$ and matrices $P'\in \Pe_{\omega\times\rho}(K)$ and $F'\in \Ef^+_{\sigma\times\omega}(L)$ such that $PF=F'P'$.
\end{lemma}
\begin{proof}
Let $\pi$ be a permutation such that $\sigma_i=\tau_{\pi(i)}$ for all $1\leq i\leq n=|\sigma|=|\tau|$, and
\[P_{ij}=\delta_{j,\pi(i)}k_i\sigma_i\quad(1\leq i,j\leq n),\]
where $k_1,\dots,k_n\in K^\times$. Write
\[F=\operatorname{diag}(F_1,\dots,F_m),\]
where each block $F_i$ is of the form (i) or (ii) in Definition~\ref{def:F}. Clearly $m=|\tau|=n$. Multiplying $F$ on the left by $P$ has the effect of multiplying each block $F_{\pi(i)}$ with $k_i$ and moving it to the $i$-th row. Clearly there is a $\omega\in \Es$ and a $Q\in \Pe_{\rho\times \omega}(K)$ such that \[PFQ=\diag(F_{\pi(1)},\dots,F_{\pi(n)}).\]
Thus $PF=F'P'$ where $F':=PFQ\in \Ef^+_{\sigma\times\omega}(L)$ and $P':=Q^{-1}\in \Pe_{\omega\times \rho}(K)$.
\end{proof}

\begin{lemma}\label{lem:commute2}
Let $F\in \Ef^-_{\sigma\times\tau}(L)$ and $P\in \Pe_{\tau\times\rho}(K)$. Then there is an $\omega\in\Es$ and matrices $P'\in \Pe_{\sigma\times\omega}(K)$ and $F'\in \Ef^-_{\omega\times\rho}(L)$ such that $FP=P'F'$.
\end{lemma}
\begin{proof}
See the proof of Lemma \ref{lem:commute1}.
\end{proof}

\begin{lemma}\label{lem:commute3}
Let $F^+_1,F^+_2\in \F^+(L)$, $P_1,P_2\in \P(L)$ and $F^-_1,F^-_2\in \F^-(L)$ such the composition $F^+_1P_1F^-_1F^+_2P_2F^-_2$ is defined. Then there is an $F^+\in \F^+(L)$, a $P\in \P(L)$ and an $F^-\in \F^-(L)$ such that 
\[F^+_1P_1F^-_1F^+_2P_2F^-_2=F^+PF^-.\]
\end{lemma}
\begin{proof}
Clearly
\begin{align*}
&F^+_1P_1F^-_1F^+_2P_2F^-_2\\
=&F^+_1P_1\hat F^+\hat F^- P_2F^-_2\\
=&F^+_1\tilde F^+\tilde P_1\tilde P_2\tilde F^- F^-_2\\
=&F^+PF^-
\end{align*}
by Proposition \ref{prop:F+F-} and Lemmas \ref{lem:commute1} and \ref{lem:commute2} where $\hat F^+, \tilde F^+\in\F^+(L)$, $\tilde P_1,\tilde P_2\in \P(L)$, $\hat F^-, \tilde F^-\in\F^-(L)$, $F^+=F^+_1\tilde F^+$, $P= \tilde P_1\tilde P_2$ and $F^-=\tilde F^- F^-_2$.
\end{proof}

\begin{proposition}\label{prop:FPnormal}
Any $A\in \F\P(L)$ can be written as $A=F^+PF^-$ where $F^+\in \F^+(L)$, $P\in \P(L)$ and $F^-\in \F^-(L)$.
\end{proposition}
\begin{proof}
Let $A\in \F\P(L)$. In view of Proposition \ref{prop:F+F-}, $A$ can be written as 
\[A=P_0F^+_1F^-_1P_1\dots F^+_nF^-_nP_n\]
where each $P_i\in \P(L)$, $F_i^+\in \F^+(L)$ and $F_i^-\in \F^-(L)$. It follows from Lemmas \ref{lem:commute1} and \ref{lem:commute2} that $A$ can be written as 
\[A=\hat F^+_1\hat P_1F^-_1\dots \hat F^+_n\hat P_n\hat F^-_n\]
where each $\hat P_i\in \P(L)$, $\hat F_i^+\in \F^+(L)$ and $\hat F_i^-\in \F^-(L)$. It follows from Lemma \ref{lem:commute3} that $A=F^+PF^-$ where $F^+\in \F^+(L)$, $P\in \P(L)$ and $F^-\in \F^-(L)$.
\end{proof}

\subsubsection{Another description of the groups $\widehat\EE_\sigma(L)$}
Let $\sigma,\tau\in \Es$ and $F\in \F^+_{\sigma\times\tau}(L)$. Then we denote the group isomorphism 
\begin{align*}
\GL_\sigma(L)&\to\GL_\tau(L)\\
A&\mapsto F^{-1}AF
\end{align*}
by $\theta_F$. Our next goal is to describe the groups $\widehat\Ee_\sigma(L)$ using the isomorphisms $\theta_F$ defined above. 

\begin{lemma}\label{lem:E-plus-conjugation}
Let \(\sigma,\tau\in\Es\) and \(F\in \Ef^+_{\sigma\times\tau}(L)\).
Then $F^{-1}\,\Ee_\sigma(L)\,F
\subseteq
\Ee_\tau(L)$.
\end{lemma}

\begin{proof}
Write $F=\operatorname{diag}(F_1,\dots,F_n)$ where $n=|\sigma|$ and each block $F_i$ is of the form (i) or (ii) in Definition~\ref{def:F}. For any $1\leq i\leq n$ let $m_i$ be the number of columns of $F_i$ (hence $m_1+\dots+m_n=|\tau|)$. 

Let $E_{kl}(x)=I_\sigma + e_{kl}(x)$, where $1\leq k,l\leq n$, be an elementary matrix in $\Ee_\sigma(L)$. Then 
\[F^{-1}E_{kl}(x)F=I_\tau + F^{-1}_{\bullet k}xF_{l\bullet}\]
where $F_{l\bullet}$ denotes the $l$-th row of $F$ and $F^{-1}_{\bullet k}$ denotes the $k$-th column of $F^{-1}$. Clearly 
\[F_{l\bullet}=(0_{m_1}, \dots, 0_{m_{l-1}}, F_l,0_{m_{l+1}}, \dots, 0_{m_{n}})\]
where $0_m$ denote the $1\times m$ vector whose components are all zero. Moreover, 
\[F^{-1}_{\bullet k}=(0_{m_1}, \dots, 0_{m_{k-1}}, F_k,0_{m_{k+1}}, \dots, 0_{m_{n}})^*\]
since $F^{-1}=F^*$. It follows that $I_\tau + F^{-1}_{\bullet k}xF_{l\bullet}$ can be written as product of elementary matrices in $\Ee_\tau(L)$.
\end{proof}

\begin{proposition}\label{prop:Etildechar}
Let $\sigma\in \Es$. Then $\widehat \EE_\sigma(L)$ is generated by the set
\[X_\sigma=\{A\in \GL_\sigma(L)\mid \exists \tau\in \Es, F\in \F^+_{\sigma\times\tau}(L):~\theta_F(A)\in\Ee_\tau(L)\}.\]
\end{proposition}
\begin{proof}
First we show that $X_\sigma\subseteq \widehat\EE_\sigma(L)$. Let $A\in\GL_\sigma(L)$ such that $\theta_F(A)\in\Ee_\tau(L)$ for some $\tau\in \Es$ and $F\in \F^+_{\sigma\times\tau}(L)$. Then $F^{-1}AF=E$ for some $E\in \Ee_\tau(L)$. It follows that $A=FEF^{-1}$ and thus $A\in \widehat \EE_\sigma(L)$.

Next we show that $\widehat \EE_\sigma(L)$ is generated by $X_\sigma$. Let $A=BEB^{-1}$ be a generator of $\widehat\EE_\sigma(L)$, where $B\in\F\P_{\sigma\times\tau}(L)$ and $E\in \Ee_\tau(L)$ for some $\tau\in \Es$. We will show that $A\in X_\sigma$. By Proposition \ref{prop:FPnormal} we can write $B=F^+PF^-$ where $F^+\in \F^+(L)$, $P\in \P(L)$ and $F^-\in \F^-(L)$. Clearly
\[E=B^{-1}AB=(F^-)^{-1}P^{-1}(F^+)^{-1}AF^+PF^-.\]
It follows from Lemma \ref{lem:E-plus-conjugation} that 
\[F^{-}E(F^-)^{-1}=P^{-1}(F^+)^{-1}AF^+P\in \EE(L)\]
(note that $(F^-)^{-1}=(F^-)^*\in \F^+(L)$). Since $\EE(L)$ is normalised by $\P(L)$, it follows that 
\[(F^+)^{-1}AF^+=\theta_{F^+}(A)\in \EE(L)\]
and thus $A\in X_\sigma$.
\end{proof}

\section{Applications to Leavitt path algebras of hypergraphs}
Throughout this section $K$ denotes a fixed field.
\subsection{Leavitt path algebras of hypergraphs}
\begin{definition}
Let $I$ and $X$ be sets. A (not necessarily injective) function $x:I\rightarrow X$, $i\mapsto x_i=x(i)$ is called a {\it family of elements in $X$ (indexed by $I$)}. We will usually denote such a family by $(x_i)_{i\in I}$. We call the family $(x_i)_{i\in I}$ {\it nonempty} if $I$ is not the empty set.
\end{definition}

\begin{definition}
A {\it (directed) hypergraph} is a quadruple $H=(H^0,H^1,r,s)$ where $H^0$ and $H^1$ are sets and $r$ and $s$ are maps associating to each $h\in H^1$ a nonempty family $r(h)=(r(h)_j)_{j\in J_{h}}$ resp. $s(h)=(s(h)_i)_{i\in I_{h}}$ of elements in $H^0$. The elements of $H^0$ are called {\it vertices} and the elements of $H^1$ {\it hyperedges}. $H$ is called {\it finite} if $H^0$ and $H^1$ are finite sets. %In this article all hypergraphs are assumed to be nonempty.
\end{definition}

\begin{definition}\label{defhlpa}
Let $H$ be a finite hypergraph. The $K$-algebra $L(H)$ presented by the generating set 
\[\{v,h_{ij},h_{ij}^*\mid v\in H^0, h\in H^1, i\in I_h,  j\in J_h\}\]
and the relations
\begin{enumerate}[(i)]
\item $uv=\delta_{uv}u\quad(u,v\in H^0)$,
\medskip
\item $s(h)_ih_{ij}=h_{ij}=h_{ij}r(h)_j,~r(h)_jh_{ij}^*=h_{ij}^*=h_{ij}^*s(h)_i\quad(h\in H^1, i\in I_h,  j\in J_h)$,
\medskip
\item $\sum\limits_{j\in J_h}h_{ij}h_{i'j}^*= \delta_{ii'}s(h)_i\quad(h\in H^1, ~i,i'\in I_h)$ and
\medskip
\item $\sum\limits_{i\in I_h}h_{ij}^*h_{ij'}= \delta_{jj'}r(h)_j\quad(h\in H^1,~j,j'\in J_h)$
\end{enumerate}
is called the {\it Leavitt path algebra} of $H$. 
\end{definition}

Note that $L(H)$ is unital with $\sum_{v\in H^0}v=1$.

\subsection{Matrix rings over hyper Leavitt path algebras as free products}

Until the end of Section 5, $H$ denotes a finite hypergraph. We write 
\[H^0=\{v_1,\dots,v_q\}\quad\text{ and }\quad H^1=\{h^{(1)},\dots,h^{(p)}\}.\]
Moreover, we assume that
\[I_{h^{(i)}}=\{1,\dots,m_i\}\quad\text{and}\quad J_{h^{(i)}}=\{1,\dots,n_i\}\]
and set
\[v_{i,j}:=s(h^{(i)})_j\quad\text{and}\quad w_{i,j'}:=r(h^{(i)})_{j'}\]
for every $1\leq i\leq p$, $1\leq j\leq m_i$ and $1\leq j'\leq n_i$.

We define the $K$-algebra $S:=K^{H^0}$ and denote for any $v\in H^0$ the element of $S$ whose $v$-th component is $1$ and whose other components are $0$ by $\alpha_v$. For any $1\leq i\leq p$ we define the finitely generated projective $S$-modules 
\[P_i:=\bigoplus_{j=1}^{m_i}\alpha_{v_{i,j}} S\quad\text{ and }\quad Q_i:=\bigoplus_{j=1}^{n_i}\alpha_{w_{i,j}} S.\]
%\[P_i:=\bigoplus_{(u,k)\in I_{h_i}}\alpha_{u}S\quad\text{ and }\quad Q_i:=\bigoplus_{(v,k)\in J_{h_i}}\alpha_{v}S.\]
Then 
\[L(H)\cong S\big\langle i_1,i_1^{-1}:\overline{P_{1}}\cong \overline{Q_{1}},\dots , i_p,i_p^{-1}:\overline{P_{p}}\cong \overline{Q_{p}}\big\rangle,\]
see \cite[\S 9.3]{Raimund2}.

Let $n=q+\sum_{i=1}^p m_i+n_i$. %Note that $n$ is large enough so that
%\[(\bigoplus_{v\in E_{\reg}} P_{v}\oplus Q_{v})\oplus (\bigoplus_{v\in E_{\sink}} vR)\]
%can be embedded as a proper direct summand in the free right $R$-module $R^n$ of rank $n$.
It follows from \cite[Theorem 4.1]{bergman74b} that
 \[\Mat_n(L(H))\cong R\big\langle i_1,i_1^{-1}:\overline{r_n(P_{1})}\cong \overline{r_n(Q_{1})},\dots , i_p,i_p^{-1}:\overline{r_n(P_{p})}\cong \overline{r_n(Q_{p})}\big\rangle,\]
where $R=\Mat_n(S)$ and for any $S$-module $M$, $r_n(M)$ denotes the $R$-module consisting of all row vectors of length $n$ in the elements of $M$, on which the elements of $R=\Mat_n(S)$ act via the usual rules for multiplying a vector by a matrix.
%, applied to the $S$-module structure of $M$.

We define the idempotent matrices
\begin{equation}e_{P_i}=\sum_{j=1}^{m_i}\epsilon^{v_{i,j}}_{z_i+j}~~(1\leq i\leq p),\quad e_{Q_i}=\sum_{j=1}^{n_i}\epsilon^{w_{i,j}}_{z_i+m_i+j}~~(1\leq i\leq p)\label{H:idemp1}
\end{equation}
and
\begin{equation}
e_i=\epsilon_{z+i}^{v_i}~(1\leq i\leq q),\quad e_0=1-\sum_{i=1}^{p}(e_{P_i}+e_{Q_i})-\sum_{i=1}^q e_i \label{H:idemp2}
\end{equation}
where $z_i=\sum_{\ell=1}^{i-1}m_\ell+n_\ell$ for $1\leq i\leq p$, $z=\sum_{\ell=1}^{p}m_\ell+n_\ell$, and $\epsilon_{k}^v$ denotes the idempotent matrix in $R$ whose entry at position $(k,k)$ is $\alpha_v$ and whose other entries are $0$. Clearly the matrices
\[e_{P_1},e_{Q_1},e_{P_2},e_{Q_2},\dots, e_{P_p},e_{Q_p},e_{1},\dots,e_q, e_0\]
form a complete set of orthogonal idempotents in $R$.

We set $R_0:=K^{2p+q+1}$ and view $R$ as an $R_0$-ring$_K$ (see \cite[\S2]{bergman74b}) via the $K$-algebra homomorphism $R_0\to R$ mapping 
\begin{align*}
(x_1,\dots,x_{2p+q+1})\mapsto&~e_{P_1}x_1+e_{Q_1}x_2+\dots+e_{P_p}x_{2p-1}+e_{Q_p}x_{2p}\\
&+e_{1}x_{2p+1}+\dots+e_{q}x_{2p+q}+e_0x_{2p+q+1}.
\end{align*}
Note that this homomorphism might not be an isomorphism since $e_0$ can be $0$. We will still denote $(1,0,\dots,0)\in R_0$ by $e_{P_1}$, $(0,1,\dots,0)\in R_0$ by $e_{Q_1}$ and so on. 

We leave it to the reader to check that $r_n(P_i)\cong e_{P_i}R$ and $r_n(Q_i)\cong e_{Q_i}R$ for any $1\leq i\leq p$. It follows that the $R$-modules $r_n(P_i),r_n(Q_i)~(1\leq i\leq p)$ are induced by $R_0$-modules:
\begin{align*}
r_n(P_i)&\cong e_{P_i}R\cong P_{0,i} \otimes_{R_0} R, \quad\text{where }P_{0,i}=e_{P_i}R_0,\\[0.2cm]
r_n(Q_i)&\cong e_{Q_i}R\cong Q_{0,i} \otimes_{R_0} R, \quad\text{where }Q_{0,i}=e_{Q_i}R_0.
\end{align*}
It follows from \cite[Theorem 3.4]{bergman74b} that
\begin{align*}
&R\big\langle i_1,i_1^{-1}:\overline{r_n(P_{1})}\cong \overline{r_n(Q_{1})},\dots , i_p,i_p^{-1}:\overline{r_n(P_{p})}\cong \overline{r_n(Q_{p})}\big\rangle\\[0.2cm]
\cong&R\amalg_{R_0}R_0\big\langle i_1,i_1^{-1}:\overline{P_{0,1}}\cong \overline{Q_{0,1}},\dots , i_p,i_p^{-1}:\overline{P_{0,p}}\cong \overline{Q_{0,p}}\big\rangle.
\end{align*}
The factor $R_0\big\langle i_1,i_1^{-1}:\overline{P_{0.1}}\cong \overline{Q_{0,1}},\dots , i_p,i_p^{-1}:\overline{P_{0,p}}\cong \overline{Q_{0,p}}\big\rangle$ is easily seen to be isomorphic to 
\[\underbrace{\Mat_2(K)\times \dots\times\Mat_2(K)}_{p\text{ factors}}\times \underbrace{K\times \dots\times K}_{q+1\text{ factors}}\]
made an $R_0$-ring via the homomorphism mapping 
\begin{align*}
(x_1,\dots,x_{2p+q+1})\mapsto (e_{11}x_1+e_{22}x_2,\dots,e_{11}x_{2p-1}+e_{22}x_{2p},x_{2p+1},\dots,x_{2p+q},x_{2p+q+1}),
\end{align*}
cf. \cite[Proof of Theorem 5.2]{bergman74b}. We have shown:

\begin{theorem}\label{thm:hlpa1}
Let $H$ be a finite hypergraph with $|H^0|=q$ and $|H^1|=p$. If $n=|H^0|+\sum_{h\in H^1}|I_h|+|J_h|$, then 
\[\Mat_n(L(H))\cong \Mat_n(K^{H^0})\amalg_{K^{2p+q+1}}(\Mat_2(K)^p\times K^{q+1}).\]
Here $\Mat_n(K^{H^0})$ and $\Mat_2(K)^p\times K^{q+1}$ are viewed as $K^{2p+q+1}$-rings via the maps
\begin{align*}
(x_1,\dots,x_{2p+q+1})\mapsto&~e_{P_1}x_1+e_{Q_1}x_2+\dots+e_{P_p}x_{2p-1}+e_{Q_p}x_{2p}\\
&+e_{1}x_{2p+1}+\dots+e_{q}x_{2p+q}+e_0x_{2p+q+1}.
\end{align*}
and 
\begin{align*}
(x_1,\dots,x_{2p+q+1})\mapsto (e_{11}x_1+e_{22}x_2,\dots,e_{11}x_{2p-1}+e_{22}x_{2p},x_{2p+1},\dots,x_{2p+q},x_{2p+q+1}),
\end{align*}
respectively, where the idempotents $e_{P_i}, e_{Q_i}, e_i$ are defined as in (\ref{H:idemp1}) and (\ref{H:idemp2}).
\end{theorem}

\begin{remark}\label{rem:hlpa1a}
If \(e_0=0\), then one can remove the last factor \(K\)
from \(K^{2p+q+1}\) and also from
\(\Mat_2(K)^p\times K^{q+1}\).
\end{remark}

\begin{remark}\label{rem:hlpa1b}
Analysing the proof of Theorem \ref{thm:hlpa1} we see that there is an isomorphism 
\[\phi:\Mat_n(K^{H^0})\amalg_{K^{2p+q+1}}(\Mat_2(K)^p\times K^{q+1})\to \Mat_n(L(H))\]
such that
\begin{align*}
\phi(e_{ij}\alpha_v)&= e_{ij}v &&\text{ for each } e_{ij}\alpha_v \text{ in }\Mat_n(K^{H^0}),\\
\phi(e^{(i)}_{11})&=\sum_{j=1}^{m_i}e_{z_i+j,z_i+j}v_{i,j}&&\text{ for } 1\leq i\leq p,\\
\phi(e^{(i)}_{22})&=\sum_{j=1}^{n_i}e_{z_i+m_i+j,z_i+m_i+j}w_{i,j} &&\text{ for } 1\leq i\leq p,\\
\phi(e^{(i)}_{12})&=\sum_{j=1}^{m_i}\sum_{j'=1}^{n_i}e_{z_i+j, z_i+m_i+j'} h^{(i)}_{jj'}&&\text{ for } 1\leq i\leq p,\\
\phi(e^{(i)}_{21})&=\sum_{j=1}^{m_i}\sum_{j'=1}^{n_i}e_{z_i+m_i+j', z_i+j} (h^{(i)}_{jj'})^* &&\text{ for } 1\leq i\leq p,\\
\phi(1^{(i)})&=e_{ z+i,z+i}v_i &&\text{ for } 1\leq i\leq q,\\
\phi(1^{(0)})&=I_n-\sum_{i=1}^p(\phi(e^{(i)}_{11})+\phi(e^{(i)}_{22}))-\sum_{i=1}^q\phi(1^{(i)}),&&
\end{align*}
where for $1\leq i\leq p$ and $k,l\in\{1,2\}$
\[e^{(i)}_{kl}=(0,\dots,0,e_{kl},0,\dots,0)\in \Mat_2(K)^p\times K^{q+1}\]
with $e_{kl}$ in the $i$-th position, for $1\leq i\leq q$, 
\[1^{(i)}=(0,\dots,0,1,0,\dots,0)\in \Mat_2(K)^p\times K^{q+1}\]
with $1$ in the $(p+i)$-th position, and
\[1^{(0)}=(0,\dots,0,1)\in \Mat_2(K)^p\times K^{q+1}.\]
\end{remark}

\begin{example}\label{ex:hlpa1}
Let \(H\) be the hypergraph with one vertex \(v\) and one hyperedge \(h\)
such that \(I_h=\{1,2\}\) and \(J_h=\{1,2,3\}\). Note that the Leavitt
path algebra \(L(H)\) is isomorphic to the Leavitt algebra \(L(2,3)\).
By Theorem~\ref{thm:hlpa1} and Remarks~\ref{rem:hlpa1a}
and~\ref{rem:hlpa1b}, there is an isomorphism
\[
\phi:\Mat_6(K)\amalg_{K^3}\bigl(\Mat_2(K)\times K\bigr)
\;\longrightarrow\;
\Mat_6(L(H))
\]
such that
\begin{align*}
\phi(e_{ij}^{(1)})&= e_{ij}~\qquad(1\le i,j\le 6)
,\\[2pt]
\phi((e^{(2)}_{11},0))&=e_{11}+e_{22},\\[2pt]
\phi((e^{(2)}_{22},0))&=e_{33}+e_{44}+e_{55},\\[2pt]
\phi((e^{(2)}_{12},0))&=
\begin{pmatrix}
0&0&h_{11}&h_{12}&h_{13}&0\\
0&0&h_{21}&h_{22}&h_{23}&0\\
0&0&0&0&0&0\\
0&0&0&0&0&0\\
0&0&0&0&0&0\\
0&0&0&0&0&0
\end{pmatrix},\\[4pt]
\phi((e^{(2)}_{21},0))&=
\begin{pmatrix}
0&0&0&0&0&0\\
0&0&0&0&0&0\\
h_{11}^*&h_{21}^*&0&0&0&0\\
h_{12}^*&h_{22}^*&0&0&0&0\\
h_{13}^*&h_{23}^*&0&0&0&0\\
0&0&0&0&0&0
\end{pmatrix},\\[4pt]
\phi((0,1))&=e_{6,6},
\end{align*}
where \(e_{ij}^{(1)}\) are the matrix units in \(\Mat_6(K)\),
\(e_{ij}^{(2)}\) are the matrix units in \(\Mat_2(K)\) and \(e_{ij}\) are the matrix units in
\(\Mat_6(L(H))\).
\end{example}

\subsection{The functor $\psi:\P(\Mat_n(L(H)))\to\P(L(H))$}

We keep the notation of \S 5.2. Set
\[
R:=\Mat_n(K^{H^0})\amalg_{K^{2p+q+1}}\bigl(\Mat_2(K)^p\times K^{q+1}\bigr)
\]
and
\(\widetilde{R}:=\Mat_n(L(H))\), and let \(\phi:R\to\widetilde{R}\) be the isomorphism
described in Remark~\ref{rem:lpa1b}. Moreover, let
\[
\mathcal{I}=\{e_{P_i},e_{Q_i},e_j,e_0\mid 1\leq i\leq p, ~1\leq j\leq q\}.
\]
Clearly \(\mathcal{J}=\phi(\mathcal{I})\) is a complete set of orthogonal idempotents
in \(\widetilde{R}\). We denote by \(\P(\widetilde{R})\) the category from Definition~\ref{def:catP}
with respect to \(\mathcal{J}\), and by \(\P(L(H))\) the category from
Definition~\ref{def:catP} with respect to \(H^0=\{v_1,\dots,v_q\}\).
The corresponding general linear groupoids are denoted by
\(\Ge(\widetilde{R})\) and \(\Ge(L(H))\).

\medskip
\noindent\textbf{Atomic decomposition of the idempotents in \(\mathcal{J}\).}
We call the idempotents
\[
\varepsilon(k,v)=e_{kk}v\in \widetilde{R}
\qquad(1\le k\le n,\ v\in H^0)
\]
\textit{atomic}. For \(u\in\mathcal{J}\), let
\(\operatorname{at}(u)\) be the ordered list of atomic idempotents whose
sum is \(u\), ordered first by row index \(k\) and, for equal \(k\), by the
fixed ordering \(v_1,\dots,v_q\) of the vertices. Denote the corresponding
ordered list of vertices by
\[
V(u)=\bigl(\operatorname{vert}(a)\bigr)_{a\in\operatorname{at}(u)},
\]
where \(\operatorname{vert}(\varepsilon(k,v))=v\).

\medskip
\noindent\textbf{Definition of \(\psi\).}
The functor $\psi:\P(\widetilde R)\to \P(L(H))$ is defined via the commutative diagram
\[
\xymatrix@C=1.2cm@R=1.5cm{
\P(\widetilde R)
  \ar[r]^{\xi_{\widetilde R}}
  \ar[d]_{\psi}
& \P'(\widetilde R)
  \ar[dr]^{\mathcal M}
&
\\
\P(L(H))
& \P'(L(H))
  \ar[l]_{\xi_{L(H)}^{-1}}
& \mathcal \P''(L(H))
  \ar[l]_{\omega},
}
\]
where $\xi_{\widetilde R}$, $\mathcal{M}$, $\omega$, $\xi_{L(H)}$ and $\P''(L(H))$ are defined as in \S 4.3. Note that \(\psi\) preserves addition of matrices.

The functor $\psi$ restricts to a functor $\Ge(\Mat_n(L(H)))\longrightarrow \Ge(L(H))$, which we denote by the same letter $\psi$.

\begin{lemma}\label{lem:hpsisurj}
The functor
\[
\psi:
\Ge(\Mat_n(L(H)))\longrightarrow \Ge(L(H))
\]
is surjective on objects and on
morphisms.
\end{lemma}

\begin{proof}
The proof is very similar to the proof of Lemma \ref{lem:psisurj} and therefore is omitted.
\end{proof}

\subsection{A generating set for $\Ge(L(H))$}
Recall that $H$ denotes a finite hypergraph. Moreover, 
\[H^0=\{v_1,\dots,v_q\},~~ H^1=\{h^{(1)},\dots,h^{(p)}\},~~I_{h^{(i)}}=\{1,\dots,m_i\},~~J_{h^{(i)}}=\{1,\dots,n_i\}\]
and 
\[v_{i,j}:=s(h^{(i)})_j,\quad w_{i,j'}:=r(h^{(i)})_{j'}\]
for every $1\leq i\leq p$, $1\leq j\leq m_i$ and $1\leq j'\leq n_i$.
We continue to write $\Ge(L(H))$ for the general linear groupoid of
$L(H)$ with respect to $H^0=\{v_1,\dots,v_q\}$.

\begin{definition}\label{def:hF}
Let $\sigma,\tau\in \Es_{H^0}$. We denote by $\Ef_{\sigma\times\tau}(L(H))$ the set of
all block-diagonal matrices
$A=\operatorname{diag}(A_1,\dots,A_\ell)\in\GL_{\sigma\times\tau}(L(H))$
such that each block $A_j$ is either of the
form
\begin{enumerate}[(i)]
\item $(v)$ where $v\in H^0$, or
\item $\begin{pmatrix}h^{(i)}_{11}&\dots&h^{(i)}_{1n_i}\\
\vdots&\ddots&\vdots\\
h^{(i)}_{m_i1}&\dots&h^{(i)}_{m_in_i}\end{pmatrix}$ where
$1\leq i\leq p$, or
\item $\begin{pmatrix}h^{(i)}_{11}&\dots&h^{(i)}_{1n_i}\\
\vdots&\ddots&\vdots\\
h^{(i)}_{m_i1}&\dots&h^{(i)}_{m_in_i}\end{pmatrix}^*=\begin{pmatrix}(h^{(i)}_{11})^*&\dots&(h^{(i)}_{m_i1})^*\\
\vdots&\ddots&\vdots\\
(h^{(i)}_{1n_i})^*&\dots&(h^{(i)}_{m_in_i})^*\end{pmatrix}$
      where $1\leq i\leq p$.
\end{enumerate}
\end{definition}

\begin{theorem}\label{thm:hlpa2}
The general linear groupoid $\Ge(L(H))$ is generated by the sets
\[
\Ee_{\sigma}(L(H)),~\Ef_{\sigma\times\tau}(L(H)),~\Pe_{\sigma\times\tau}(K)
\quad (\sigma,\tau\in\Es_{H^0}).
\]
\end{theorem}
\begin{proof}
By Theorem~\ref{thm:hlpa1} there is an isomorphism
\[
\phi : R\to \Mat_n(L(H)),
\qquad
R = \Mat_n(K^{H^0})\amalg_{K^{2p+q+1}}\bigl(\Mat_2(K)^p\times K^{q+1}\bigr).
\]
Write $R_0=K^{2p+q+1}$, $R_1=\Mat_n(K^{H^0})$ and $R_2=\Mat_2(K)^p\times K^{q+1}$. We will consider only the case that $e_0\neq 0$. The case $e_0=0$ can be handled similarly, but one should delete the last factor \(K\) from \(R_0\) and the last
factor \(K\) from \(R_2\) (see Remark \ref{rem:hlpa1a}). Both $R_1$ and $R_2$ are faithful $R_0$-rings, and $R$ is their coproduct in the category of $R_0$-rings.  

Considering the $V$-monoids of $R_0$, $R_1$ and $R_2$ we see that every finitely generated projective module over one of these rings is pseudo‑free. Hence Condition (P) is satisfied by Remark \ref{rem:condP}. Theorem~\ref{thm:main} therefore tells us that the general linear
groupoid $\Ge(R)$ with respect to
\[\mathcal{I}=\{e_{P_i},e_{Q_i},e_j,e_0\mid 1\leq i\leq p,\ 1\leq j\leq q\}\]
is generated by the subgroupoids $\Ge(R,\mu)\;(\mu=0,1,2)$.
Recall that for $\mu\in\{0,1,2\}$, $\Ge(R,\mu)$ denotes the subgroupoid of $\Ge(R)$ generated by
the sets $\GL_{\sigma\times\tau}(R_\mu)$ and 
$\T_\sigma(R,\mu)$ where $\sigma,\tau\in \Es_\I$.

Clearly the isomorphism $\phi:R\to\Mat_n(L(H))$ in Remark \ref{rem:hlpa1b} induces an isomorphism
\[
\phi:\Ge(R)\to\Ge(\Mat_n(L(H))),
\]
where $\Ge(\Mat_n(L(H)))$ denotes the general linear groupoid of
$\Mat_n(L(H))$ with respect to $\mathcal{J}=\phi(\mathcal{I})$.
Let
\[
\psi:\Ge(\Mat_n(L(H)))\to\Ge(L(H))
\]
be the functor defined in \S 5.3. Since $\psi$ is surjective by
Lemma~\ref{lem:hpsisurj}, the composition
\[
\psi\circ\phi : \Ge(R)\to\Ge(L(H))
\]
is surjective. It therefore suffices to show that all generators of
$\Ge(R)$ are mapped by $\psi\circ\phi$ into the subgroupoid
$\H\subseteq\Ge(L(H))$ generated by the sets
$\Ee_{\sigma}(L(H))$, $\Ef_{\sigma\times\tau}(L(H))$ and $\Pe_{\sigma\times\tau}(K)
$ where $\sigma,\tau\in\Es_{H^0}$. The rest of the proof is almost identical to the corresponding part of the proof of Theorem \ref{thm:lpa2}, and therefore is omitted.
\end{proof}

\begin{remark}\label{rem:hone-vertex}
It is easy to see that if the hypergraph \(H\) has exactly one
vertex, then every generalised permutation matrix in
\(\Ge(L(H))\) can be written as a product of elementary matrices and
one invertible diagonal matrix with entries in \(K^\times\).
Consequently, for one-vertex hypergraphs,
Theorem~\ref{thm:hlpa2} remains true if one replaces
\(\Pe_{\sigma\times\tau}(K)\) by \(\De_\sigma(K)\).
\end{remark}

\begin{remark}\label{rem:iso-vertex}
Suppose that $H$ is the hypergraph with one vertex $v$ and one hyperedge $h$ such that $|I_h|=m$ and $|J_h|=n$ where $m,n\geq 2$. Then $L(H)$ is isomorphic to the Leavitt algebra $L(m,n)$. Let $\sigma=(v,\dots,v)\in (H^0)^\ell$ for some $\ell< \min(m,n)$. The $\V$-monoid of $L(H)$ has the presentation $\langle v\mid mv=nv\rangle$ (see \cite{Raimund2}). Hence all sets $\GL_{\sigma\times\tau}(L(H))$ and $\GL_{\tau\times\sigma}(L(H))$, where $\tau\neq\sigma$, are empty. It follows from Theorem~\ref{thm:hlpa2} and Remark~\ref{rem:hone-vertex}
that $\GL_{\sigma}(L(H))\cong \GL_{\ell}(L(H))$ is generated by $\De_\sigma(K)$ and $\Ee_\sigma(L(H))$. In particular, $\GL_1(L(H))=K^\times$.
\end{remark}

In the following we denote $L(H)$ by $L$.

\begin{definition}\label{def:hfatALL}
We denote by 
\begin{itemize}
\item $\EE(L)$ the subgroupoid of
$\Ge(L)$ generated by the sets $\Ee_{\sigma}(L)$ where $\sigma\in\Es$,
\item $\F\P(L)$ the subgroupoid of $\Ge(L)$ generated by all sets $\Pe_{\sigma\times\tau}(K)$ and $\Ef_{\sigma\times\tau}(L)$ where $\sigma,\tau\in\Es$,
\item $\widehat{\EE}(L)$ the subgroupoid of $\Ge(L)$ generated by all matrices of the form $AEA^{-1}$ where $A\in \F\P(L)$ and $E\in \EE(L)$ are composable. 
\end{itemize}
For $\mathbf{C}\in\{\F\P,\widehat\EE\}$ and $\sigma,\tau\in \Es$ we set $\mathbf{C}_{\sigma\times\tau}(L):=\mathbf{C}(L)\cap \GL_{\sigma\times \tau}(L)$ and $\mathbf{C}_{\sigma}(L):=\mathbf{C}_{\sigma\times\sigma}(L)$.
\end{definition}

\begin{corollary}\label{cor:hlpa2}
Let $\sigma\in \Es$. Then every matrix $A\in \GL_\sigma(L)$ can be written as $A=BC$ where $B\in \widehat{\EE}_\sigma(L)$ and $C\in\F \P_{\sigma}(L)$.
\end{corollary}
\begin{proof}
By Theorem \ref{thm:hlpa2} a matrix $A\in \GL_\sigma(L)$ can be written as a composition of matrices from $\widehat{\EE}(L)$ and $\F\P(L)$. Clearly $\widehat{\EE}(L)$ is normalised by $\F\P(L)$, i.e. if $X\in \F\P(L)$ and $Y\in \widehat{\EE}(L)$ are composable, then $XYX^{-1}\in  \widehat{\EE}(L)$. It follows that $A$ can be written as $A=BC$ where $B\in \widehat{\EE}_\sigma(L)$ and $C\in \F\P_{\sigma}(L)$.
\end{proof}

\section*{Conflict of interest}
The author declares that there is no conflict of interest.

\section*{Data availability}
No datasets were generated or analysed.

\end{document}